\documentclass[12pt]{amsart}
\usepackage[english]{babel}
\usepackage{amsmath,amsthm}
\usepackage{amsfonts}
\usepackage{enumerate}
\usepackage{mathtools}
\usepackage{graphicx}
\usepackage{color}
\usepackage{subcaption}

\usepackage[all]{xy}
\usepackage{thmtools, thm-restate}
\usepackage{overpic}

\usepackage[top=1.5in, bottom=1.5in, left=1.5in, right=1.5in]{geometry}

\newtheorem{thm}{Theorem}[section]
\newtheorem{cor}[thm]{Corollary}
\newtheorem{lem}[thm]{Lemma}
\newtheorem{prop}[thm]{Proposition}

\theoremstyle{definition}
\newtheorem{defn}[thm]{Definition}

\newtheorem{question}{Question}

\theoremstyle{remark}
\newtheorem{rem}[thm]{Remark}
\newtheorem{exam}[thm]{Example}
\numberwithin{equation}{section}

\newcommand{\spin}{\ifmmode{\rm Spin}\else{${\rm spin}$\ }\fi}
\newcommand{\spinc}{\ifmmode{{\rm Spin}^c}\else{${\rm spin}^c$}\fi}

\newcommand{\Z}{\mathbb{Z}}

\newcommand{\Q}{\mathbb{Q}}
\newcommand{\C}{\mathbb{C}}
\newcommand{\R}{\mathbb{R}}
\newcommand{\M}{\mathcal{M}}

\newcommand{\RowMatrix}[1]{
\begin{matrix}
#1 & \dotsb & #1
\end{matrix}
}

\DeclareMathOperator{\im}{im}

\DeclareMathOperator{\tor}{tor}
\DeclareMathOperator{\Hom}{Hom}
\DeclareMathOperator{\lcm}{lcm}
\DeclareMathOperator{\inter}{int}
\DeclareMathOperator{\lk}{lk}

\newcommand{\note}[1]{\textbf{{\color{red} #1 }}}

\begin{document}

\title{Doubly slice Montesinos links}%

\author{Duncan McCoy}%
\address {Universit\'{e} du Qu\'{e}bec \`{a} Montr\'{e}al}
\email{mc\_coy.duncan@uqam.ca}

\author{Clayton McDonald}%
\address {Boston College}
\email{mcdonafi@bc.edu}
\date{}%

\begin{abstract}
This paper compares notions of double sliceness for links. The main result is to show that a large family of 2-component Montesinos links are not strongly doubly slice despite being weakly doubly slice and having doubly slice components. Our principal obstruction to strong double slicing comes by considering branched double covers. To this end we prove a result classifying Seifert fibered spaces which admit a smooth embeddings into integer homology $S^1 \times S^3$s by maps inducing surjections on the first homology group. A number of other results and examples pertaining to doubly slice links are also given.
\end{abstract}

\maketitle
\section{Introduction}
A knot $K \subseteq S^3$ is said to be {\em doubly slice} if it arises as the intersection of an unknotted 2-sphere in $S^4$ and the equatorial $S^3$.\footnote{Here and throughout the paper we work exclusively in the smooth category. That is, all manifolds and embeddings are assumed to be smooth.} A 2-sphere in $S^4$ is unknotted if it bounds an embedded ball. The notion of double slicing for knots was introduced by Fox \cite{Fox62} and has been studied using a whole host of different techniques  (see, amongst others, \cite{Sumner71, Ruberman83, Friedl04, Donald2015Embedding, Meier15}). However there are several natural ways this notion can be extended to links. Following \cite{McDonald2019doubly}, a 2-component link $L$ is said to be {\em strongly doubly slice} if it arises as the intersection of an unlink of 2-spheres in $S^4$ with the equatorial $S^3$ and that $L$ is {\em weakly doubly slice} if it arises as the intersection of an unknotted 2-sphere with the equatorial $S^4$. Note that a weak double slicing on $L$ induces a quasi-orientation on $L$, so it is natural to consider being weakly doubly slice as a property of a link $L$ with a quasi-orientation. The main result of this paper is to show that a large family of 2-components Montesinos links are weakly doubly slice with both quasi-orientations and have doubly slice components, but are not strongly doubly slice.
\begin{restatable}{thm}{thmMain}\label{thm:montesinos}
Let $L$ be the Montesinos link
\[
L=\M \left(0; \frac{p_1}{q_1}, \dots,\frac{p_k}{q_k}, -\frac{p_k}{q_k}, \dots, -\frac{p_1}{q_1}  \right),
\]
where at most one of the $p_i$ is even and for all $i$ we have  $|p_i/q_i|\geq 2$. Then $L$ is a 2-component link which is weakly doubly slice with both quasi-orientations and both components are doubly slice knots. However, if there are $i, j$ such that $\gcd(p_i, p_j)>1$ and $\frac{p_i}{q_i} \neq \pm \frac{p_j}{q_j}$, then $L$ is not strongly doubly slice.
\end{restatable}
Our conventions on Montesinos links are laid out at the beginning of Section~\ref{sec:Montesinos}.
\begin{exam} The link $\M(0; 5,\frac52, -\frac52, -5)$, depicted in Figure~\ref{fig:monty_example}, is not strongly doubly slice, despite being weakly doubly slice with both quasi-orientations and having unknotted components.
\end{exam}
Although Theorem~\ref{thm:montesinos} is not the first known instance that examples of links which are weakly doubly slice but not strongly doubly slice have been produced, it provides the first examples where the link is weakly doubly slice with both quasi-orientations and has doubly slice components, but is still not strongly doubly slice. In \cite{McDonald2019doubly}, the second author exhibited links which are weakly doubly slice with one quasi-orientation, but not strongly doubly slice. However these examples are only known to be weakly slice with one quasi-orientation and could not be strongly doubly slice as their components were not doubly slice as knots.
\begin{figure}[!ht]
  \begin{overpic}[width=250pt]{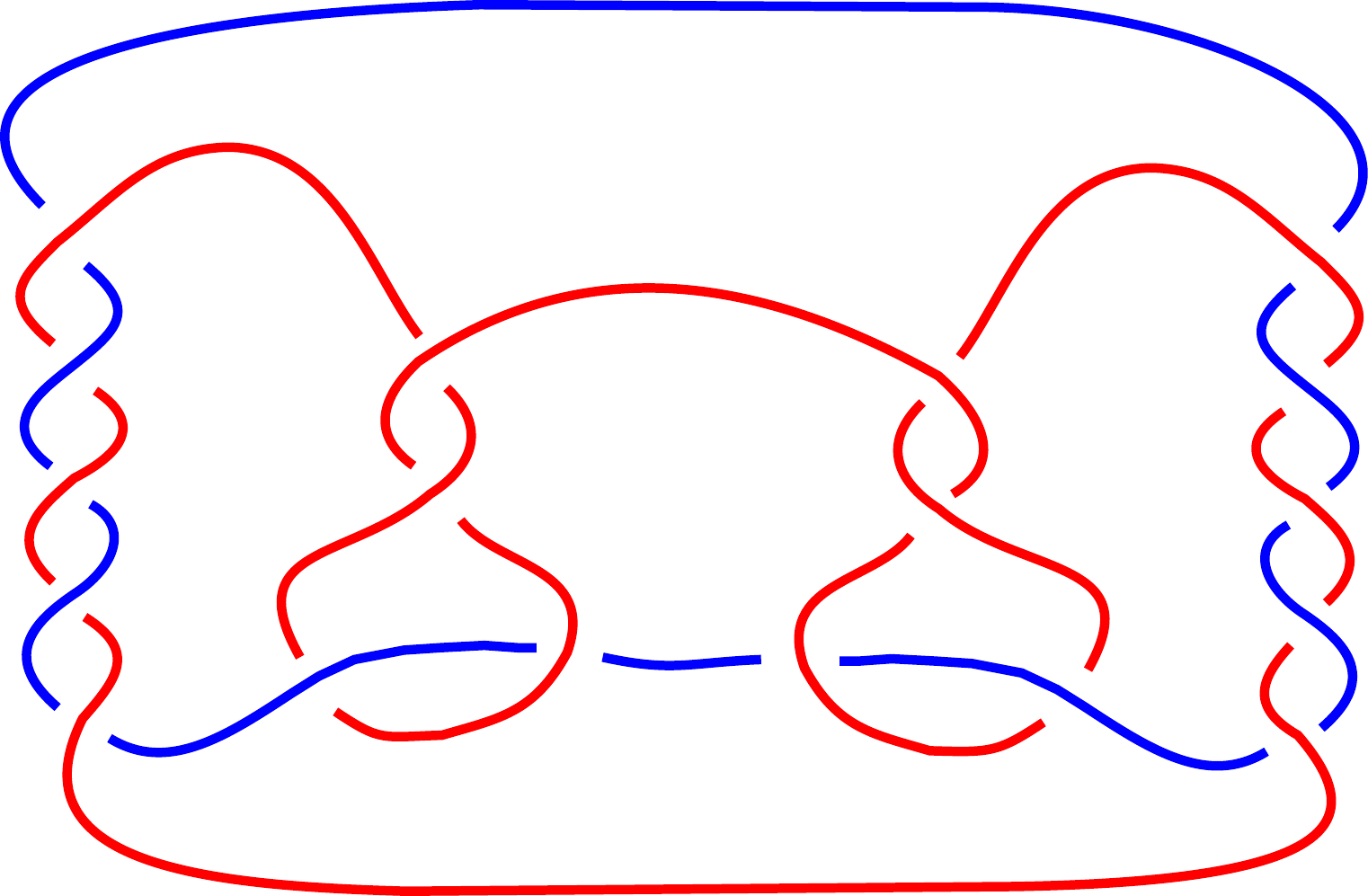}
  \end{overpic}
  \caption{The Montesinos link $\M(0; 5,\frac52, -\frac52, -5)$.}
  \label{fig:monty_example}
  \end{figure}
We note that any Montesinos link of the form 
\[
L=\M \left(0; \frac{p_1}{q_1}, \dots,\frac{p_k}{q_k}, -\frac{p_k}{q_k}, \dots, -\frac{p_1}{q_1}  \right)
\]
is isotopic to one where the $p_i/q_i$ satisfy $|p_i/q_i|\geq 2$. Thus the condition in Theorem~\ref{thm:montesinos} that the $|p_i/q_i|\geq 2$ for all $i$ is not actually a restriction, it is merely a convention that allows us to state our conclusions more concisely.

Furthermore using the work of Aceto-Kim-Park-Ray \cite{Acetoandco}, we show that for four stranded 2-component pretzel links the notions of slice and weakly doubly slice coincide. 

\begin{restatable}{thm}{fourstrand}\label{thm:4strand}
Let $L$ be a pretzel link with two components and four strands, then the following are equivalent:
\begin{enumerate}[(i)]
    \item $L$ takes the form $L=P(a,b,-b,-a)$ with at most one of $a,b$ even;
    \item $L$ is slice;
    \item $L$ is weakly doubly slice with at least one quasi-orientation; and
    \item $L$ is weakly doubly slice with both quasi-orientations.
\end{enumerate}
\end{restatable}

\subsection{Double slicing with more than two components}
So far we have only discussed double slicing for links with two components. For links with more components there are obviously more possibilities for double slicings based on the number unknotted, unlinked spheres used in the slicing.

We propose the following general definition for double slicing of links.
\begin{defn}\label{def:dbly_slice_link}
Let $L=L_1 \sqcup \dots \sqcup L_r \subseteq S^3$ be a coloured oriented link with $r$ colours and a quasi-orientation\footnote{Here, a {\em quasi-orientation} on a coloured link is a choice of orientation up to overall reversal on each monochromatic sublink. Thus an $n$-component link coloured with $r$ colours has $2^{n-r}$ possible quasi-orientations}. We say that $L$ is doubly slice if there is $S= S_1 \sqcup \dots \sqcup S_r\subseteq S^4$ such that $S$ isotopic to an unlink of 2-spheres in $S^4$ such that each $S_i$ intersects the equatorial $S^3$ in the monochromatic sublink $L_i$ and induces the chosen quasi-orientation.
\end{defn}
This definition we naturally encapsulates definitions for strong and weak double slicing. A link is strongly doubly slice if and only if it is doubly slice with the colouring in which every component has a distinct colour. A quasi-oriented link is weakly doubly slice if it is doubly slice when the link is coloured using a unique colour.


Although this paper is primarily focused on the case of two component links we also consider some examples relevant to this more general definition. We provide examples of three component links that are weakly slice with exactly one quasi-orientation. For every $n$, we exhibit examples of prime non-split links with $n$ components which are strongly doubly slice.

We discuss several constructions of double slicings. Most notably we prove the following which can be viewed as an extension of Zeeman's result that $K\# -K$ is doubly slice for any knot $K\subseteq S^3$ \cite{Zeeman1965twist_spin}.
\begin{restatable}{thm}{thmTangleDouble} \label{thm:doubling_tangle}
Let $T$ be a tangle consisting of $n$ arcs embedded in $B^3$. Then the link obtained by doubling $T$ is an $n$-component link which is weakly doubly slice with all quasi-orientations.
\end{restatable}

\subsection{Embedding Seifert fibered spaces.}
Our primary obstruction to being strongly doubly slice comes from considering the double branched cover (see Lemma~\ref{lem:dbc_embedding}).  This leads us to study when a Seifert fibered space $Y$ can be smoothly embedded into $S^1 \times S^3$ by a map which induces the following surjection: $H_1(Y;\Z) \rightarrow H_1(S^1 \times S^3;\Z)$. We approach this question using an obstruction derived from Donaldson's diagonalization theorem, which allows us to study to the slightly weaker condition of when $Y$ can be embedded into an integral homology $S^1 \times S^3$. In this paper, we use $Y\cong S^2(e; \frac{p_1}{q_1}, \dots, \frac{p_k}{q_k})$ to denote the space obtained by surgery as illustrated in Figure~\ref{fig:sfs_as_surgery}. We also recall the definition of expansion for Seifert fibered spaces \cite[Definition~1.5]{Issa2018embedding}. Given $Y\cong S^2(e; \frac{p_1}{q_1}, \dots, \frac{p_k}{q_k})$, we say that $Y'$ is obtained from $Y$ by {\em expansion} if it can be written the form
\[
Y' \cong S^2\left(e; \frac{p_1}{q_1}, \dots, \frac{p_k}{q_k}, \frac{p_j}{q_j}, -\frac{p_j}{q_j}\right)
\]
for some $j$ in the range $1\leq j \leq k$. Our obstruction yields the following theorem.
\begin{restatable}{thm}{thmSFembedding}\label{thm:embedding}
Let $Y$ be a Seifert fibered space over base surface $S^2$. Then the following are equivalent:
\begin{enumerate}[(i)]
    \item\label{it:embedding} there exists a smooth $\Z H_*(S^1\times S^3)$ $Z$ such that $Y$ embeds into $Z$ and the induced map $H_1(Y;\Z)\rightarrow H_1(Z;\Z)$ is surjective;
    \item\label{it:expansion} $Y$ is obtained by expansion from a Seifert fibered space which bounds a smooth $\Z H_* (S^1 \times B^3)$;
    \item\label{it:explicit_Description} $Y$ is homeomorphic to a space of the form
    \[
Y\cong S^2 \left(0; \left\{\frac{p_1}{q_1}, -\frac{p_1}{q_1}\right\}^{\geq 1}, \dots, \left\{\frac{p_k}{q_k}, -\frac{p_k}{q_k}\right\}^{\geq 1}  \right),
\]
where $p_i/q_i \geq 2$ for all $i$ and $\gcd(p_i, p_j)=1$ if $i\neq j$.
\end{enumerate}
\end{restatable}
Here the notation $\left\{\frac{p_i}{q_i}, -\frac{p_i}{q_i}\right\}^{\geq 1}$ is used to denote the fact that there is at least one copy of this pair of fractions in the description of $Y$.
\begin{figure}[!ht]
  \begin{overpic}[width=250pt]{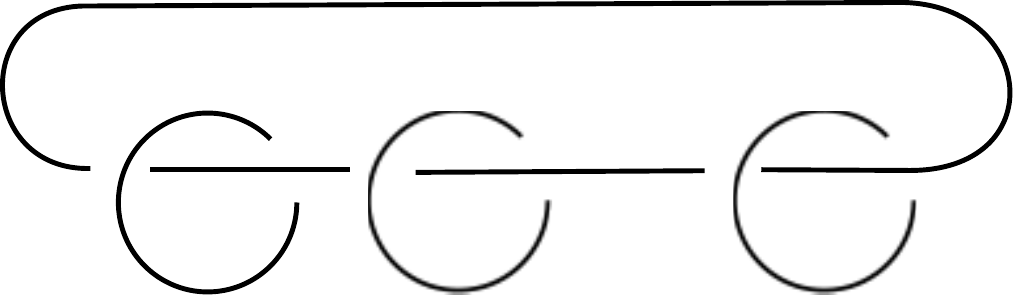}
    \put (-5, 18) {\large $e$}
    \put (60, 5) {\large $\dots$}
    \put (18, -5) {\large $\frac{p_1}{q_1}$}
    \put (43, -5) {\large $\frac{p_2}{q_2}$}
    \put (80, -5) {\large $\frac{p_k}{q_k}$}
  \end{overpic}
  \vspace{0.5cm}
  \caption{Surgery presentation of the Seifert fibered space $S^2(e; \frac{p_1}{q_1}, \ldots, \frac{p_k}{q_k})$.}
  \label{fig:sfs_as_surgery}
\end{figure}
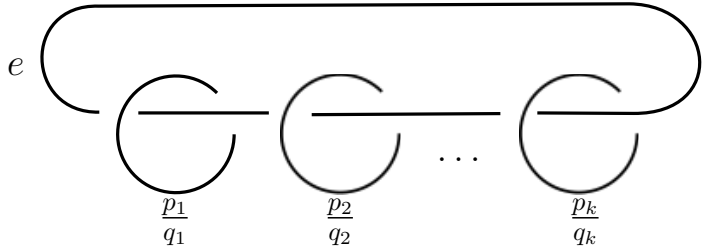

\subsection{Structure} In Section~\ref{sec:constructions}, we discuss constructions of double slicings. The discussion of obstructions is given in Section~\ref{sec:obstructions}. In Section~\ref{sec:Montesinos}, we combine the material form the preceding to sections to discuss the double sliceness of Montesinos links. This section treats Theorem~\ref{thm:embedding} as black box. Finally in Section~\ref{sec:SFS_embedding}, we turn our attention to embedding of Seifert spaces into homology $S^1 \times S^3$s and conclude the paper with a proof of Theorem~\ref{thm:embedding}.

\subsection*{Acknowledgements}
The second author would like to thank his advisor, Josh Greene, for his support and guidance.
\section{Constructing double slicings}\label{sec:constructions}
In this section we consider various constructions which allow us to produce weakly and strongly doubly slice links. 
\begin{figure}[!ht]
  \begin{overpic}[width=300pt]{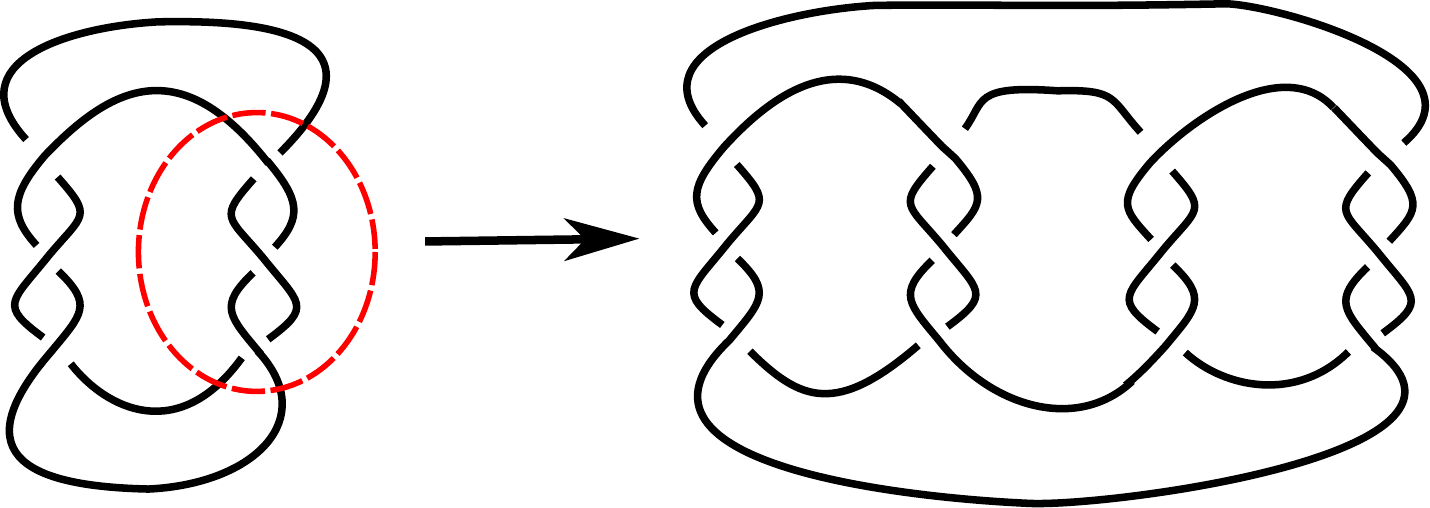}
    \put (5,-5) {$U\sqcup U$}
    \put (60,-5) {$P(3,-3,3,-3)$}
  \end{overpic}
    \vspace{0.5cm}
  \caption{Folding (Lemma~\ref{lem:folding}) a copy of the unlink to show that $P(3,-3,3,-3)$ is strongly doubly slice.}
  \label{fig:3333_example_folding}
\end{figure}
\subsection{Folding}
The following lemma, originally due to Issa, is useful for constructing new double slicings \cite[Lemma~4.9.2]{IssaThesis}. We refer this lemma as the folding construction, since the link $L'$ is obtained by ``folding'' the equatorial $S^3$ (cf. Figure~\ref{fig:Ahmad_lem_schematic}). As an example, Figure~\ref{fig:3333_example_folding} shows how we can apply Lemma~\ref{lem:folding} to a diagram of the unlink with two components to show that the pretzel knot $P(3,-3,3,-3)$ is strongly doubly slice. 
\begin{lem}[Folding construction]\label{lem:folding}
Let $L$ be a link with a planar diagram $D_L$ which contains a disk $D$ intersecting the link in a tangle $T$. Let $L'$ be the link obtained by modifying $D_L$ inside $D$ as shown in Figure~\ref{fig:Ahmad_lemma_statement}. If $L$ arises as transverse intersection $L= F \cap S^3$, between a surface $F\subset S^4$ and an equatorial $S^3$, then we can realize $L'$ as a transverse intersection $L'=F'\cap S^3$, where $F'$ is ambiently isotopic to $F$.
\end{lem}

\begin{figure}[!ht]
  \begin{overpic}[width=350pt]{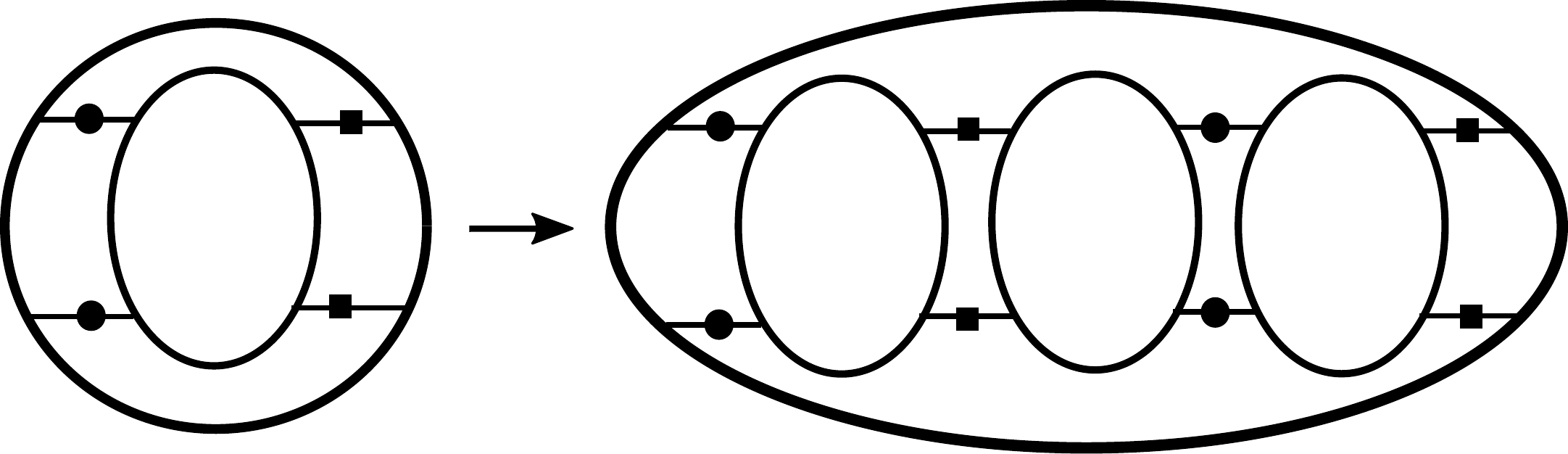}
    \put (11, -5) {\large $D_L$}
    \put (68,-5) {\large $D_{L'}$}
 \put (11, 12) {\LARGE $T$}
 \put (51, 12) {\LARGE $T$}
 \put (68, 12) {\LARGE $\widetilde{T}$}
 \put (83, 12) {\LARGE $T$}
  \put (5, 13) {\large $\vdots$}
 \put (22, 13) {\large $\vdots$}
 \put (45, 13) {\large $\vdots$}
\put (61, 13) {\large $\vdots$}
 \put (77, 13) {\large $\vdots$}
 \put (93, 13) {\large $\vdots$}
  \end{overpic}
  \vspace{0.5cm}
  \caption{The link obtained by folding $T$. The tangle $\widetilde{T}$ is obtained by rotation $T$ by $\pi$ about an axis vertical in the plane of the diagram and then changing all crossings. If $L$ is oriented, then we orient $\widetilde{T}$ by reversing the orientation on each component. If $L$ is coloured, then we colour $
  \widetilde{T}$ with the colouring it inherits naturally. The annotations on the strands on the strands entering and leaving $T$ are to illustrate that the tangle $\widetilde{T}$ has been rotated.}
  \label{fig:Ahmad_lemma_statement}
\end{figure}
\begin{figure}[!ht]
  \begin{overpic}[width=110pt]{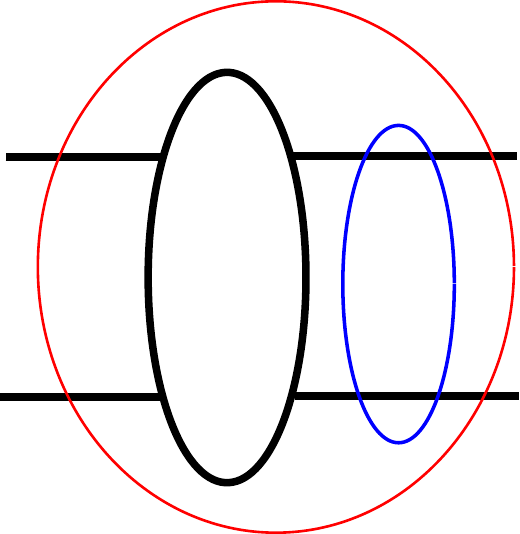}
    \put (35,41) {\LARGE $T$}
    \put (-3,50) {\color{red}$R$}
    \put (60,73) {\color{blue} $B$}
  \end{overpic}
  \caption{The intersection of the spheres $R$ and $B$ with the diagram $L$.}
  \label{fig:Ahmad_lem_spheres}
\end{figure}

\begin{figure}[!ht]
  \begin{overpic}[width=350pt]{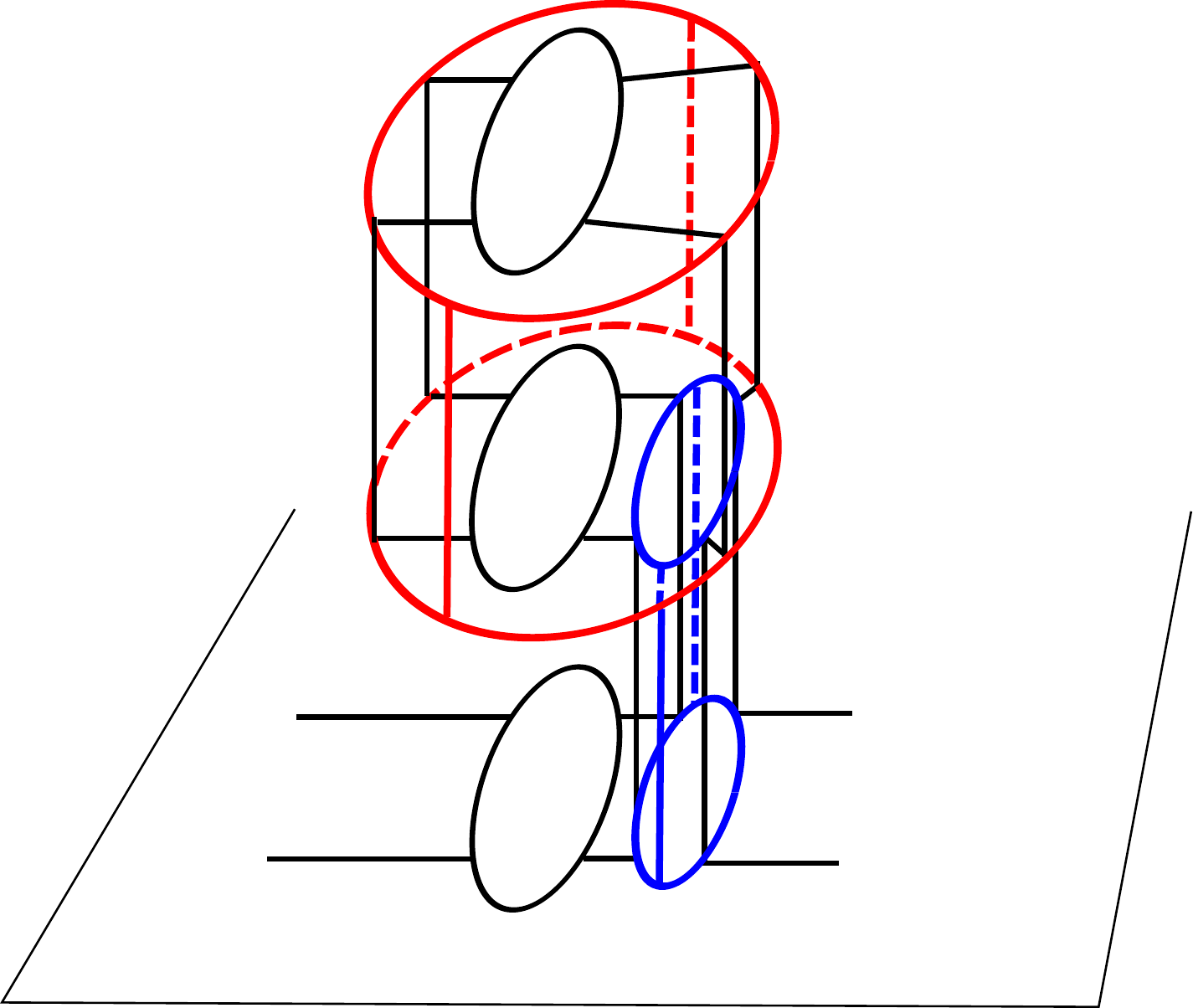}
  \put (0, 20) {\LARGE $S^3$}
    \put (44, 17) {\LARGE $T$}
    \put (44, 70) {\LARGE $T$}
    \put (44,43) {\LARGE \reflectbox{$\widetilde{T}$}}
  \end{overpic}
  \caption{A schematic of the new new copy of $S^3$.}
  \label{fig:Ahmad_lem_schematic}
\end{figure}
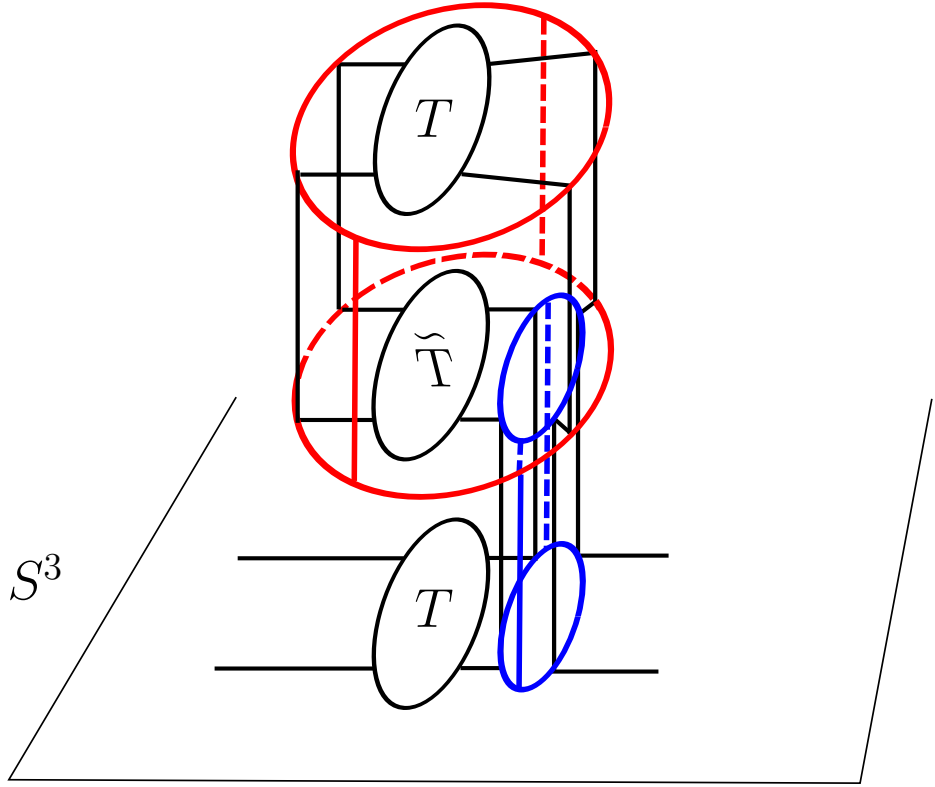

\begin{proof}
Rather than isotope the surface we instead choose a new copy of $S^3$, which is isotopic to the original equatorial $S^3$, but intersects $F$ in the desired link. The reader should refer to Figure~\ref{fig:Ahmad_lem_schematic} for a schematic illustrating the idea of the construction. Consider the starting $S^3$ as the boundary of a copy of $B^4$ in $S^4$. Since $F$ intersects $S^3$ transversely, we can parametrize a neighbourhood of the equatorial $S^3$ as $S^3 \times [-1,1]$, where the equatorial $S^3$ is $S^3\times \{0\}$, the $B^4$ intersects this neighbourhood as $B^4 \cap S^3 \times [-1,1]= S^3\times [-1,0]$ and the surface $F$ intersects this neighbourhood as $F \cap S^3 \times [-1,1]= L\times [-1,1]$. Now choose two 3-balls in $S^3\times \{0\}$, $R$ and $B$, which intersect the plane of the diagram $D_L$ as illustrated in Figure~\ref{fig:Ahmad_lem_spheres} and $B\subseteq R$. Consider the set
\[
X= B^4 \cup \left(B \times \left[0, \frac13 \right]\right) \cup \left(R \times \left[\frac13, \frac23 \right]\right)
\]
This $X$ is isotopic to the original $B^4$, and one can check that the boundary $\partial X$ intersects $F$ in a copy of the link $L'$.
\end{proof}

We can use this proposition to generate a variety of new examples of weak and strong double slicings:
\begin{prop}
There exist prime non-split strong double slicings of any number of components.
\end{prop}
\begin{proof}

Take $n$ cyclically ordered crossingless unknots as in  Figure~\ref{fig:cyclicunknots} and add 3 nested Reidemeister-2 moves between each adjacent pair of unknots. Then by applying Proposition~\ref{lem:folding} along each set of positive crossings as in Figure~\ref{fig:rmove}, we get a strongly doubly slice link (Figure~\ref{fig:stronglink}) such that every pair of adjacent components forms the pretzel link $P(3,-3,3,-3)$. 
If we consider the graph with vertices corresponding to components of this link and edges corresponding to nontrivial linking between the components, the resulting graph would be a cycle. This means that the link must not be split, as the graph would be disconnected. This also means that the link is prime, as it has unknotted components so it cannot have any knot summands, and if it had link summands then removing the connect sum component would make the link split. Because the link's associated graph is 2-connected, the link is prime.

\begin{figure}[h]
    
     \begin{subfigure}{0.4\textwidth}
        \centering
        \includegraphics[width = 0.8\textwidth]{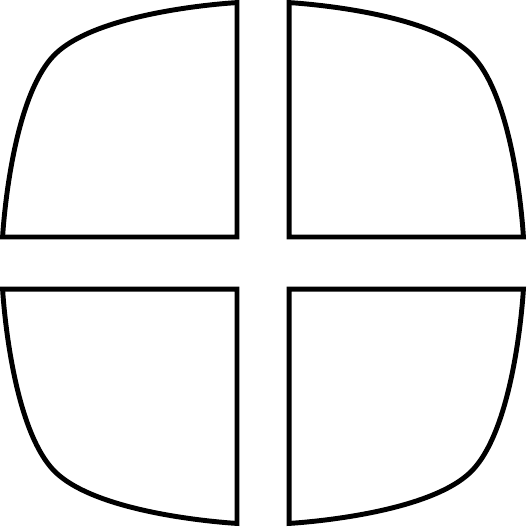}
        \caption{cyclic unknots, with n = 4}
        \label{fig:cyclicunknots}
    \end{subfigure}
    \hspace{1cm}
    \begin{subfigure}{0.4\textwidth}
        \centering
        \includegraphics[width=0.8\textwidth]{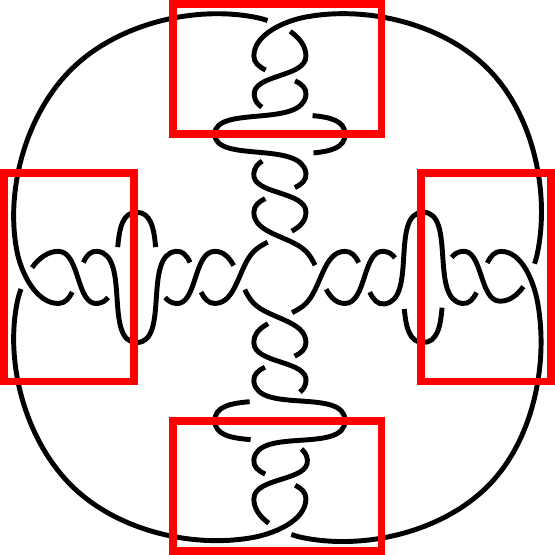}
        \caption{After Reidemeister 2 moves, with highlighted folding boxes.}
        \label{fig:rmove}
    \end{subfigure}
    \hspace{1cm}
    \begin{subfigure}{0.4\textwidth}
        \centering
        \includegraphics[width=0.8\textwidth]{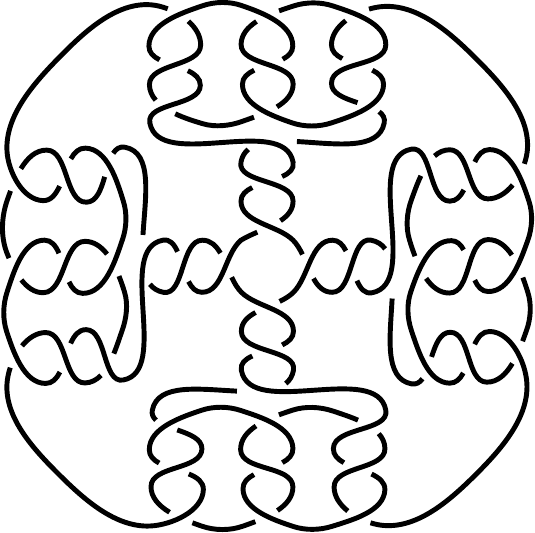}
        \caption{The resulting 4 component link.}
        \label{fig:stronglink}
    \end{subfigure}
\end{figure}
\end{proof}


\subsection{Tangle Doubling}
Using a similar construction to folding, we can construct weak double slicings with all quasi-orientations, among other things. Since the double of $B^3$ is $S^3$, given a tangle $T$ in $B^3$, we can construct the double of $T$ in $S^3$, where the double naturally inherits orientations and colourings from the initial tangle $T$. An schematic of this operation in terms of diagrams is illustrated in Figure~\ref{fig:doubling_T}.

\begin{figure}[!ht]
  \begin{overpic}[width=300pt]{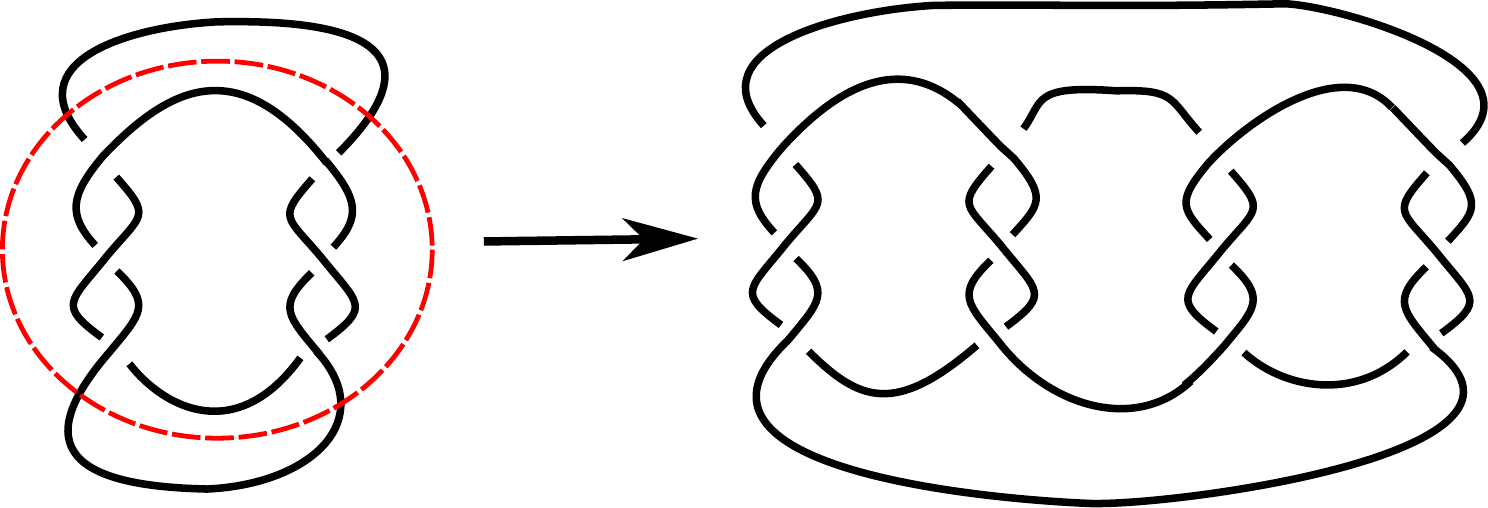}
    \put (5,-5) {$U\sqcup U$}
    \put (60,-5) {$P(3,-3,3,-3)$}
  \end{overpic}
    \vspace{0.5cm}
  \caption{Showing that $P(3,-3,3,-3)$ is strongly doubly slice by doubling a subtangle of the unlink (cf. Proposition~\ref{prop:tangle_doubling}).}
  \label{fig:3333_example2}
\end{figure}
\begin{prop}\label{prop:tangle_doubling}
Let $L$ be a link in $S^3$ and suppose that there is a 3-ball intersecting the link in a tangle $T$. Let $L'$ be the link obtained by doubling $T$. If $L$ arises as transverse intersection $L= F \cap S^3$, between a surface $F\subset S^4$ and an equatorial $S^3$, then we can realize $L'$ as a transverse intersection $L'=F'\cap S^3$, where $F'$ is isotopic to $F$.
\end{prop}
\begin{figure}[!ht]
  \begin{overpic}[width=300pt]{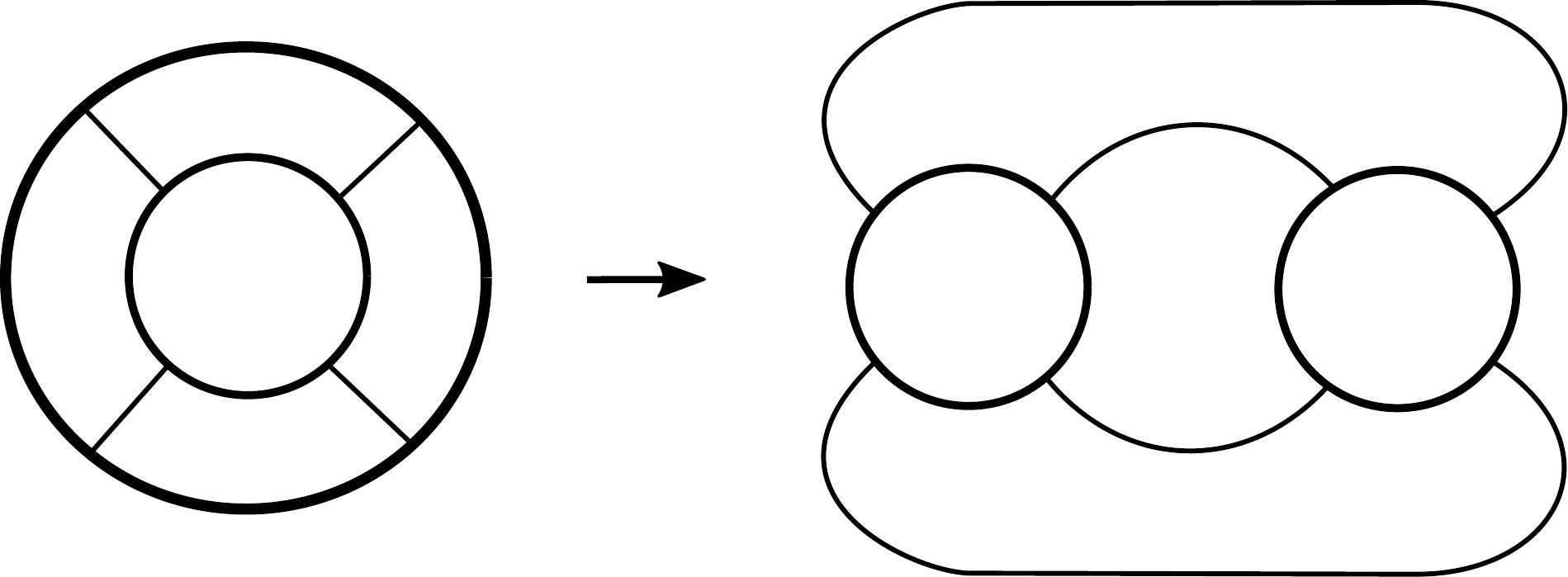}
  \put (0, 35) {\large $D_L$}
    \put (13, 17) {\LARGE $T$}
     \put (60, 17) {\LARGE $T$}
    \put (87,17) {\LARGE $\widetilde{T}$}
  \end{overpic}
  \caption{The link obtained by doubling $T$.}
  \label{fig:doubling_T}
\end{figure}
\begin{proof}
Again, rather than isotope the surface we instead choose a new copy of $S^3$, which is isotopic to the original equatorial $S^3$, but intersects $F$ in the desired link. Since $F$ intersects $S^3$ transversely, we can parametrize a neighbourhood of the equatorial $S^3$ as $S^3 \times [-1,1]$, where the equatorial $S^3$ is $S^3\times \{0\}$. Let $B\subseteq S^3$ be a ball which intersects $L$ in the tangle $T$. Consider the 4-ball given by $B \times [0,\frac12]$, the boundary of this ball is copy of $S^3$ which intersects $F$ in a copy of $L'$. 
\end{proof}

We prove Theorem~\ref{thm:doubling_tangle} as a corollary of the above statement.

\thmTangleDouble*
\begin{proof}
Start with the tangle $T$ lying inside a ball $B$. The tangle has $2n$ end points. For any choice of orientations on the arcs of $T$ we can connect $n-1$ of these endpoints by arcs outside of $B$ in way consistent with the orientations on the arcs to form a single connected component. We can view this arc as a knot $K$ with a trivial arc removed. Attach an arc between the final pair of end points to form $K\# -K$. This is doubly slice by Zeeman's twist spinning construction \cite{Zeeman1965twist_spin}. Thus all possible orientations of $T$ arise as subtangles of a doubly slice knot. Thus by Proposition~\ref{prop:tangle_doubling} we have that the double of $T$ is weakly doubly slice with all quasi-orientations.
\end{proof}



\section{Obstructing double slicing}\label{sec:obstructions}
First, we can obtain restrictions on doubly slice links just from considering how the slicing sphere can intersect $S^3$.
\begin{lem}\label{lem:linking_mono}
Suppose that $L$ is an oriented $n$-component link. There is a sphere $S\subseteq S^4$ such that $L=S \cap S^3$ if and only if there are two partitions of the components $\mathcal{P}_1$ and $\mathcal{P}_2$ into sublinks with the following properties:
\begin{enumerate}[(i)]
    \item\label{it:partitition_size} $|\mathcal{P}_1|+|\mathcal{P}_2|=n+1$
    \item\label{it:surfaces} For $k=1,2$, we have that for any distinct $L\in \mathcal{P}_k$ bounds a planar surface $\Sigma_L$ in $B^4$. Moreover if $L,L'\in \mathcal{P}_k$ are distinct sublinks then the surfaces $\Sigma_L$ and $\Sigma_{L'}$ are disjoint.
    \item\label{it:graph_cond} Let $G$ be the bipartite graph with vertex set $\mathcal{P}_1\cup \mathcal{P}_2$ with an edge between $L\in \mathcal{P}_1$ and $L'\in \mathcal{P}_2$ for each common component. Then $G$ is connected (equivalently a tree). 
\end{enumerate}
\end{lem}
\begin{proof}
Suppose that $L=S\cap S^3$ for some $S\subseteq S^4$. Then $S\setminus L$ separates $S$ into $n+1$ components. And $S^3$ separates $S^4$ into two 4-balls $B_1$ and $B_2$. We take one partition for each ball, with the surfaces classes corresponding to connected component of $S \setminus L$ lying in that ball. Conditions \eqref{it:partitition_size} and \eqref{it:surfaces} are evident. Condition \eqref{it:graph_cond} is precisely the condition required to glue the surfaces from \eqref{it:surfaces} together to form a sphere in $S^4$.
\end{proof}
Lemma~\ref{lem:linking_mono} has some easy but useful consequences. 
\begin{rem}\label{rem:useful} Suppose that a link $L$ is doubly slice with at least two components.
\begin{enumerate}
    \item\label{it:wk_dbl_slice} If we have a two component link $L=L_1\cup L_2$, which is weakly doubly slice, then one of the partitions arising from Lemma~\ref{lem:linking_mono} has to be a partition of the form $\{\{L_1\},\{L_2\}\}$. Hence the link $L$ has to be slice
    \item\label{it:slice_comp} More generally, one can see that if a class in one of the partitions $\mathcal{P}_1$ or $\mathcal{P}_2$ is a singleton, then the corresponding component of $L$ is a slice knot. Since every tree has at least two leaves, we see that at least two classes in the partitions $\mathcal{P}_1$ and $\mathcal{P}_2$ must be singletons. Thus at least two components of a weakly doubly slice link $L$ are slice as knots. 
    \item\label{it:linking} Finally, note that if $L, L' \in \mathcal{P}_k$, then $\lk (L, L')=0$, since they bound disjoint embedded surfaces in $B^4$.
\end{enumerate}
\end{rem}
Next we establish the obstruction to double slicing which arises by considering double branched covers.

\begin{lem}\label{lem:dbc_embedding}
Let $L$ be a coloured link with $n$ colours which is doubly slice. Then there is an embedding of $\Sigma(L)$ into $\#_{n-1} S^1 \times S^3$ such that the induced map 
\[H_1(\Sigma(L),\Z) \rightarrow H_1(\#_{r-1} S^1\times S^3,\Z)\]
is a surjection.
\end{lem}


\begin{proof}

We first show that the double cover of $S^4$ branched over an $n$ component unlink of $S^2$s is $\#_{n-1} S^1 \times S^3$. Note first that the double cover of $S^4$ branched over a single unknotted $S^2$ is $S^4$. Since the complement of an unknotted $S^2$ in $S^4$ is $S^1 \times B^3$, the double cover of the complement is again $S^1 \times B^3$ and filling in the branching locus gives a copy of $S^4$. If we puncture $S^4$ and take the branched cover over that same unknotted sphere, the result is a copy of $S^4$ with two punctures which we can identify with $S^3\times I$. The double cover of $S^4$ branched over a two component unlink can be decomposed into double covers of two punctured $S^4$'s, each branched over unknotted spheres and then glued to each other along their common boundary. This corresponds to gluing two $S^3 \times I$'s together to form $S^3 \times S^1$. We can extend this to more components by noting that connect summing along a single component of a branch locus produces the connect sum of the resulting branched covers, and the $n$-component sphere-unlink is the $n-1$-fold connect sum of two component sphere unlinks.
If we restrict our branched cover to the $S^3$ that hits our unknotted sphere link in $L$, we see that there is a natural embedding of $\Sigma(L)$ into $\#_{n-1} S^1 \times S^3$. 


Let $X$ be the complement of $n$-component sphere-unlink. $\pi_1(X)$ can naturally be identified with $F_n$, generated by the meridians of each component. Let $M$ be the cover corresponding to the kernel of the map from $\pi_1(X)$ to $\Z/2\Z$ given by sending every meridian to the nontrivial element. $\pi_1(M)$ is therefore generated by products of any two meridians.

If we take the product of two meridians of the link components in $S^3$, this will differ in $\pi_1(X)$ from the generators of $\pi_1(M)$ by a choice of whisker to each of the link components. This difference vanishes when we abelianize, meaning that the inclusion map on $H_1$ is surjective for the double covers of the complements.

The fundamental group $\pi_1(\#_{n-1} S^1 \times S^3)$ is a quotient $\pi_1(M)$ obtained by attaching the $S^2 \times B^2$'s corresponding to the branch locus, each of which quotients by the square of a meridian. If we look at the restriction of this attachment to the 3-manifold, we see that the regular double cover is quotiented by the squares of the meridians in the same way, so the inclusion map on $H_1$ is also surjective on the level of branched covers.


\end{proof}

\subsection{Weak double slicing with only one quasi-orientation}
Using Lemma~\ref{lem:folding} and Lemma~\ref{lem:linking_mono}, we find examples of three component links that are weakly doubly slice with exactly one quasi-orientation.

\begin{prop}
The pretzel link $P(2n+1, -2n, 2n, -2n)$ is a three component link that is weakly doubly slice with precisely one quasi-orientation.
\end{prop}
\begin{proof}
The pretzel link $P(2n+1,-2n,2n,-2n)$ can be constructed from $P(2n+1,-2n)$ by folding along the $-2n$ strand as in Figure~\ref{fig:3222_example}. Since $P(2n+1,-2n)$ is the unknot, this implies $P(2n+1,-2n,2n,-2n)$ is weakly doubly slice with the quasi-orientation induced by $P(2n+1,-2n)$. To obstruct weak double slicings with the other quasi-orientations we employ Lemma~\ref{lem:linking_mono}. First, note that the three components of $P(2n+1,-2n,2n,-2n)$ are two unknots $U_1$ and $U_2$ and one copy of the torus knot $T = T_{2,2n+1}$, which is not slice. If there were a weak double slicing, then there would be partitions $\mathcal{P}_1, \mathcal{P}_2$ of the components as in Lemma~\ref{lem:linking_mono}. Since $T$ is not slice, we see that the partitions must take the form $\mathcal{P}_1=\{\{T\cup U_1\}, \{U_2\}\}$ and $\mathcal{P}_2=\{\{T\cup U_2\}, \{U_1\}\}$, cf. Remark~\ref{rem:useful}\ref{it:slice_comp}. This corresponds to the slicing sphere being cut into two discs and two annuli by the equatorial $S^3$. As in the reasoning of Remark~\ref{rem:useful}\ref{it:linking}, this implies that the components must be oriented so that 
\begin{equation}\label{eq:orientation_cond}
\lk(T\cup U_1, U_2)=\lk(T\cup U_2, U_1)=0.
\end{equation}
The linking numbers satisfy
\[|\lk(T,U_1)|=|\lk(T,U_2)|=|\lk(U_1,U_2)|=n\] 
and so up to reorienting every component at once, there is only one way to orient the components in a manner satisfying \eqref{eq:orientation_cond}, namely taking the orientations so that
\[
\lk(T,U_2)=\lk(T,U_1)=-\lk(U_1,U_2).
\]
This shows that a weak double slicing can exist with exactly one quasi-orientation, as required.
\end{proof}
\begin{figure}[!ht]
  \begin{overpic}[width=300pt]{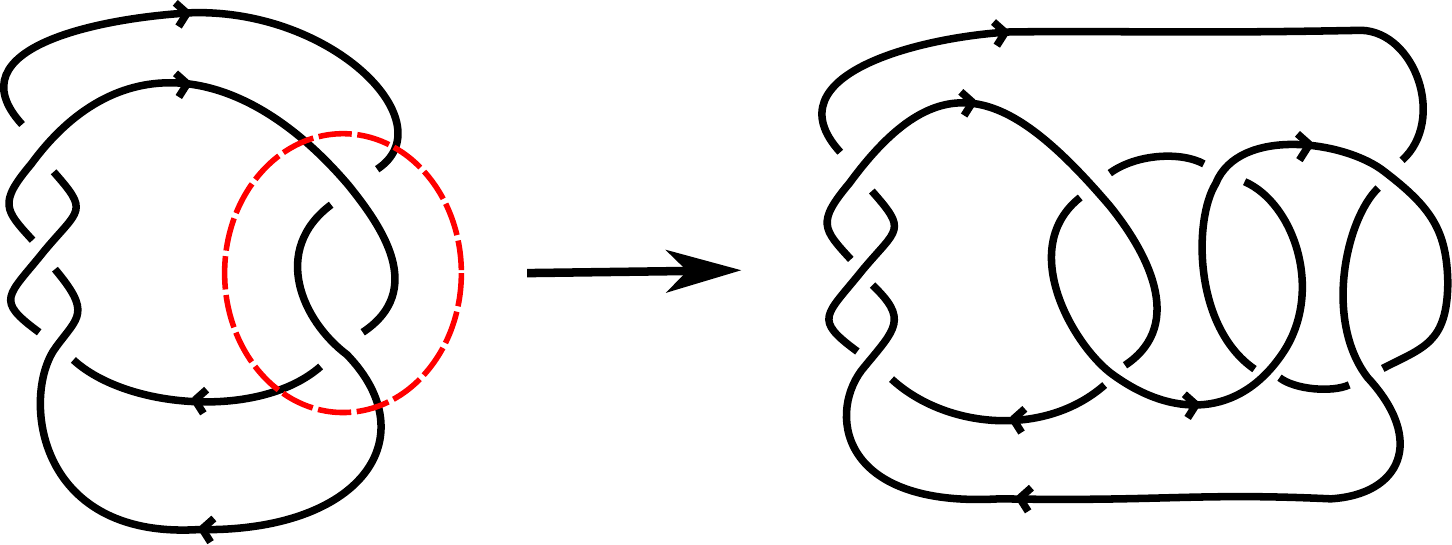}
    \put (15,-5) {$U$}
    \put (70,-5) {$P(3,-2,2,-2)$}
  \end{overpic}
    \vspace{0.5cm}
  \caption{Showing that $P(3,-2,2,-2)$ is weakly doubly slice.}
  \label{fig:3222_example}

\end{figure}

\section{Double slicing for Montesinos links}\label{sec:Montesinos}
First we lay out some conventions concerning rational tangles and Montesinos links. For any $p/q\in \Q$, the $p/q$-rational tangle is the tangle built up from the tangles $1/0$ and $0/1$ using the relationships depicted in Figure~\ref{fig:rational_tangles}.

\begin{figure}[!ht]
  \begin{overpic}[width=300pt]{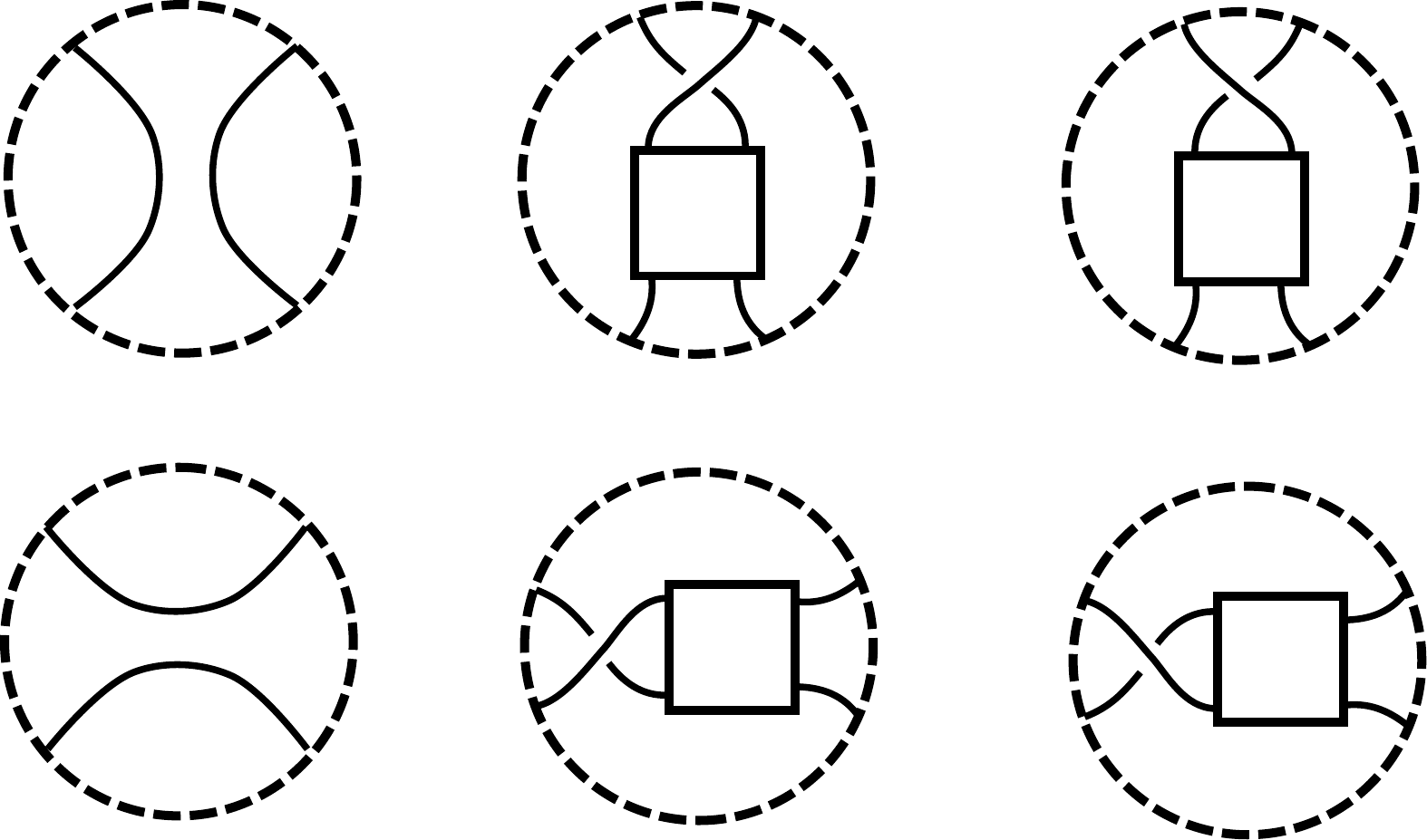}
    \put (10,29) {\large $\frac01$}
    \put (10,-5) {\large$\frac10$}
    \put (46,29) {\large $\frac{p+q}{q}$}
    \put (46,-5) {\large $\frac{p}{p+q}$}
    \put (85,29) {\large $\frac{p-q}{q}$}
    \put (85,-5) {\large $\frac{p}{p-q}$}
    \put (46,42) {$p/q$}
    \put (48,12) {$p/q$}
    \put (85,42) {$p/q$}
    \put (87,12) {$p/q$}
  \end{overpic}
    \vspace{0.7cm}
  \caption{Building up rational tangles.}
  \label{fig:rational_tangles}
\end{figure}
We take the $\M(e; \frac{p_1}{q_1}, \dots, \frac{p_k}{q_k})$ to be the link illustrated in Figure~\ref{fig:Montesinos_diagram}. Since the $\frac{p}{p+q}$ rational tangle is obtained from the $\frac{p}{q}$ rational tangle by introducing a crossing on the side, we see the one can perform an isotopy (a flype) to show that the links
\[\M\left(e; \frac{p_1}{q_1}, \dots , \frac{p_k}{q_k}\right)
\quad\text{and}\quad
\M\left(e+1; \frac{p_1}{q_1}, \dots,\frac{p_i}{p_i+q_i},\dots , \frac{p_k}{q_k}\right)\]
are isotopic.
With these conventions the double branched cover of $\M(e; \frac{p_1}{q_1}, \dots, \frac{p_k}{q_k})$ is the Seifert fibered space $S^2(e; \frac{p_1}{q_1}, \dots, \frac{p_k}{q_k})$.
\begin{figure}[!ht]
  \begin{overpic}[width=300pt]{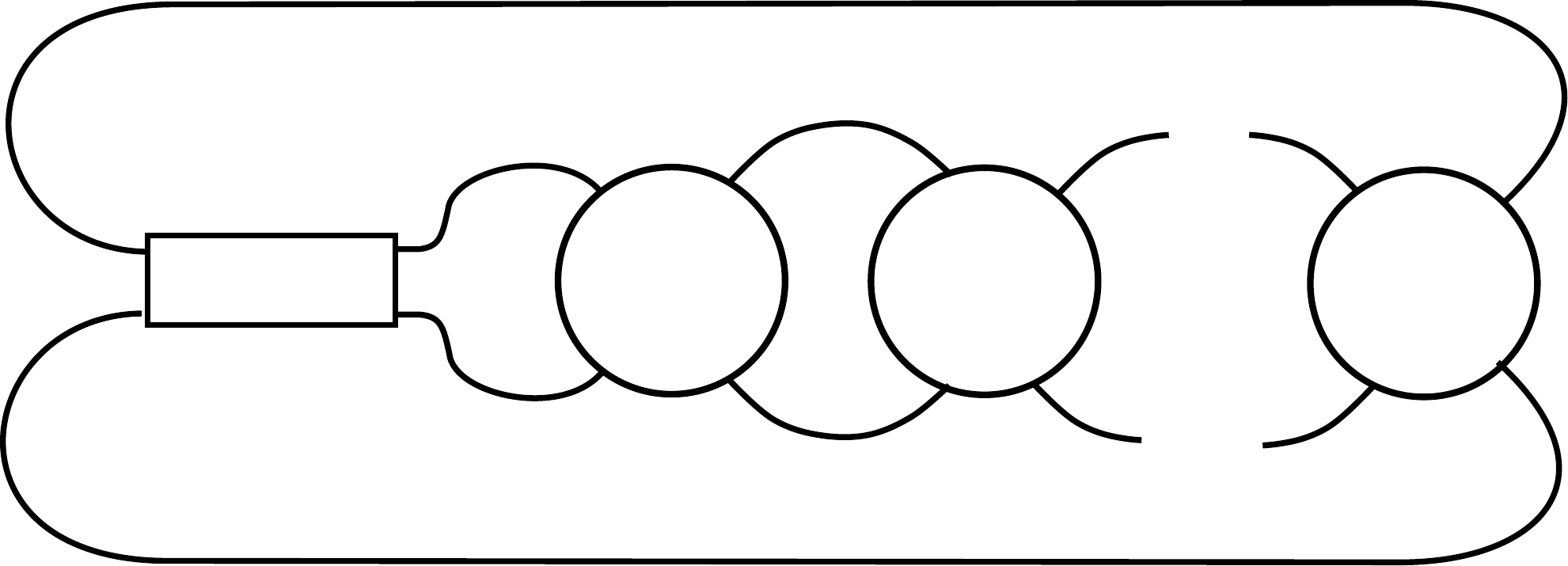}
    \put (41,17) {\large $\frac{p_1}{q_1}$}
    \put (61,17) {\large $\frac{p_2}{q_2}$}
    \put (89,17) {\large $\frac{p_k}{q_k}$}
    \put (14,17) {$e$}
    \put (74,16.5) {\Large $\dots$}
  \end{overpic}
  \caption{The Montesinos link $\M(e; \frac{p_1}{q_1}, \dots, \frac{p_k}{q_k})$.}
  \label{fig:Montesinos_diagram}
\end{figure}

\begin{prop}\label{prop:weak_slicing_Montesinos}
Let $L$ be a Montesinos link of the form
\[
L=\M\left(0;
\frac{p_1}{q_1}, \dots, \frac{p_k}{q_k},-\frac{p_k}{q_k}, \dots,- \frac{p_1}{q_1}
\right),
\]
where at most one of the $p_i$ are even. Then $L$ is a two component link which is weakly doubly slice with both quasi-orientations and both components of $L$ are doubly slice knots.
\end{prop}
\begin{proof}
Consider the tangle $T$ in Figure~\ref{fig:Montesinos_tangle}. Every rational tangle consists of two arcs. It follows, by an easy induction, for example, from the conventions laid out in Figure~\ref{fig:rational_tangles} that the configuration of these endpoints of these arcs is governed by the parity of $p$ and $q$ as shown in Figure~\ref{fig:rational_parity}. From this it follows that the tangle $T$ has no closed components if at most one of the $p_i$ is even. Thus if at most one of the $p_i$ are even we can apply Proposition~\ref{prop:tangle_doubling} to show that double of $T$ is 2-component link which is weakly slice with both quasi-orientations. However the double of $T$ is precisely the Montesinos link $L=\M(0; \frac{p_1}{q_1}, \dots, \frac{p_k}{q_k},-\frac{p_k}{q_k}, \dots,- \frac{p_1}{q_1})$, which is the desired link. Since $L$ is obtained as the double of a tangle both its components take the form $K\#-K$ for some knot $K$.
\end{proof}
\begin{figure}[!ht]
  \begin{overpic}[width=300pt]{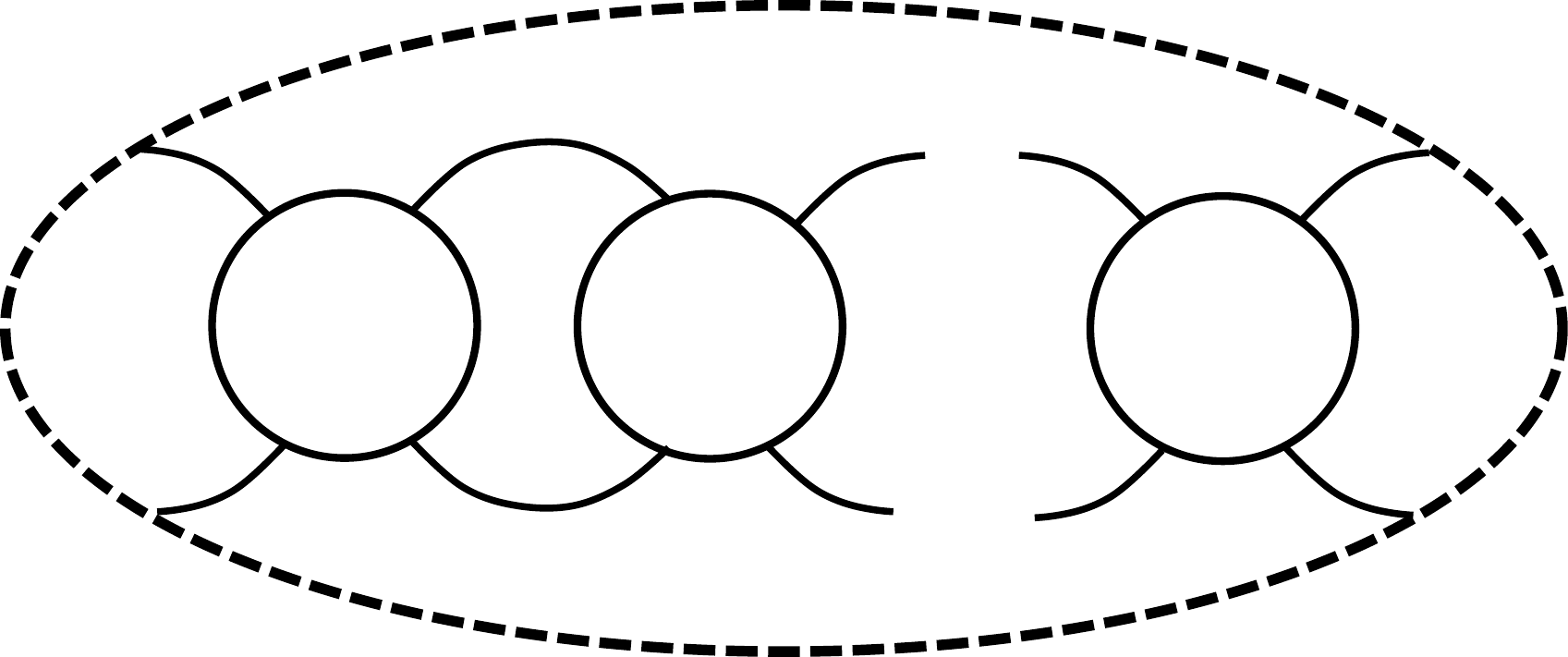}
    \put (21,19) {\Large $\frac{p_1}{q_1}$}
    \put (42,19) {\Large $\frac{p_2}{q_2}$}
    \put (76,19) {\Large $\frac{p_k}{q_k}$}
    \put (57,19) {\LARGE $\dots$}
  \end{overpic}
  \caption{The tangle $T$ to be doubled in the proof of Proposition~\ref{prop:weak_slicing_Montesinos}.}
  \label{fig:Montesinos_tangle}
\end{figure}
\begin{figure}[!ht]
  \begin{overpic}[width=300pt]{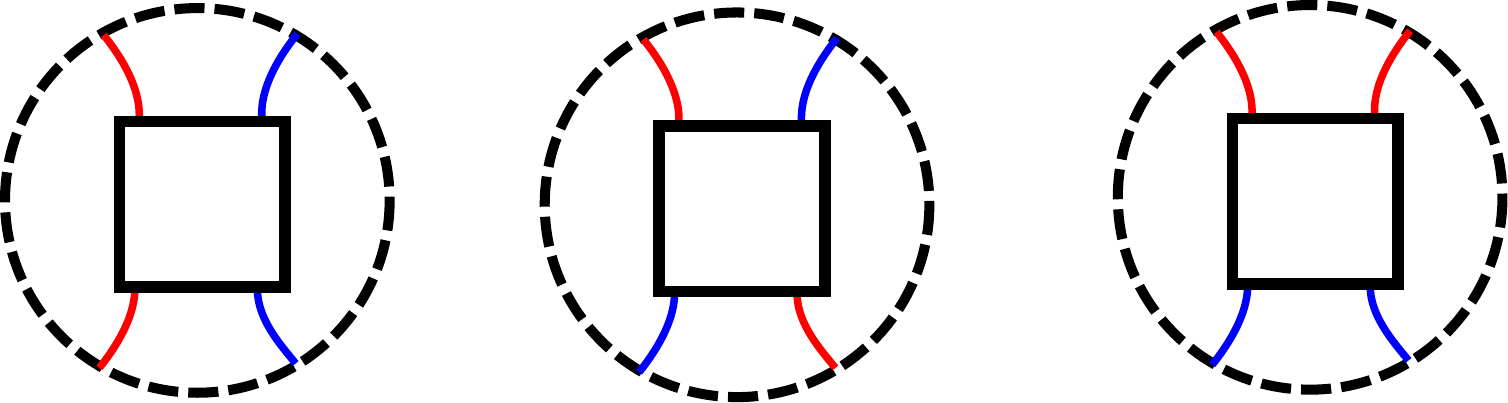}
    \put (13,12) {\Large $\frac{p}{q}$}
    \put (47,12) {\Large $\frac{p}{q}$}
    \put (85,12) {\Large $\frac{p}{q}$}
    \put (0,-5) {$p$ even}
    \put (43,-5) {$p$ odd}
    \put (85,-5) {$p$ odd}
     \put (0,-10) {$q$ odd}
    \put (43,-10) {$q$ odd}
    \put (82,-10) {$q$ even}
  \end{overpic}
  \vspace{1cm}
  \caption{The configuration of endpoints on the boundary of a rational tangle depend on the parities of $p$ and $q$.}
  \label{fig:rational_parity}
\end{figure}

Using Proposition~\ref{prop:weak_slicing_Montesinos}, we can prove Theorem~\ref{thm:4strand}. 
\fourstrand*
\begin{proof}
The implication $(iv)\Rightarrow (iii)$ is trivial. The implication $(iii)\Rightarrow (ii)$ follows from Remark~\ref{rem:useful}\eqref{it:wk_dbl_slice}. The results of \cite{Acetoandco} imply that the slice 4-stranded 2-component pretzel links are in the form demanded by $(i)$. Finally, the implication $(i)\Rightarrow (iv)$ is comes from Proposition~\ref{prop:weak_slicing_Montesinos}, which constructs weak double slicings with both quasi-orientations for these links.
\end{proof}

Additionally, we can now prove Theorem~\ref{thm:montesinos} assuming Theorem~\ref{thm:embedding}.
\thmMain*
\begin{proof}
Proposition~\ref{prop:weak_slicing_Montesinos} exactly shows that $L$ is a two component link which is slice with both quasi-orientations and both components are doubly slice. Now suppose that $L$ is strongly doubly slice. Lemma~\ref{lem:dbc_embedding} implies that the double branched cover $\Sigma(L)$ admits a smooth embedding into $S^1\times S^3$ and the induced map on homology $H_1(\Sigma(L);\Z) \rightarrow H_1(S^1\times S^3;\Z)$ is surjective. The double branched cover of $L$ is the Seifert fibered space
\[
Y\cong S^2\left(0; \frac{p_1}{q_1}, -\frac{p_1}{q_1}, \dots, \frac{p_k}{q_k}, -\frac{p_k}{q_k}\right).
\]
Thus Theorem~\ref{thm:embedding} implies that if there are $i$ and $j$ such that $p_i$ and $p_j$ are not coprime, then the set $\{\frac{p_i}{q_i}, -\frac{p_i}{q_i}\}$ has to equal $\{\frac{p_j}{q_j}, -\frac{p_j}{q_j}\}$. Thus we have $\frac{p_i}{q_i}= \pm\frac{p_j}{q_j}$, as required.
\end{proof}

\begin{lem}
If $p$ is odd, then the Montesinos link
\[\M\left(0;\frac{p}{q}, -\frac{p}{q}, \dots, \frac{p}{q}, -\frac{p}{q}\right)\]
is a strongly doubly slice two component link.
\end{lem}
\begin{proof}
The link $\M\left(0; \frac{p}{q}, -\frac{p}{q}\right)$ is the unlink on two components and thus is strongly doubly slice. By repeated applications of Lemma~\ref{lem:folding} we see that $L=\M\left(0;\frac{p}{q}, -\frac{p}{q}, \dots, \frac{p}{q}, -\frac{p}{q}\right)$ arises as the intersection of a two component unlink in $S^4$ with an equatorial $S^3$. Since $L$ has two components, this implies that $L$ is strongly doubly slice. 
\end{proof}

Using the above corollary along with Lemma \ref{lem:folding}, we can prove a more general statement with restricted coefficients:

\begin{cor}
If $p$ is odd then every mutant of $L = \M(0;\frac{p}{q},-\frac{p}{q}, \dots, \frac{p}{q},-\frac{p}{q})$ is weakly doubly slice with both quasi-orientations.
\end{cor}
\begin{proof}
For brevity we will use the following notation:
\[J(n_1, \dots, n_{2k})=\M\left(0;\left\{\frac{p}{q}\right\}^{\times n_1},\left\{-\frac{p}{q}\right\}^{\times n_2},\dots  ,\left\{\frac{p}{q}\right\}^{\times n_{2k-1}},\left\{-\frac{p}{q}\right\}^{\times n_{2k}}\right).\]
The mutants of $L$ are obtained by permuting its parameters, and so any such mutant $L'$ can written in the form $L'=J(n_1, \dots, n_{2k})$, where the sum of $n_i$ for $i$ odd equals the sum of the $n_i$ for $i$ even. Note that this representation will be unique up to cyclic permutation.

From here the proof will proceed by induction on the value $k$. For $k=1$, we have a link of the form $J(n_1, n_1)$ and the result follows immediately by Proposition~\ref{prop:weak_slicing_Montesinos}. For a larger tuple, consider the smallest $n_i$. By cyclically reordering and reflecting if necessary, we can assume that $n_{2k-1}$ is this minimal value. Since $n_{2k-1}\leq n_{2k}$ and $n_{2k-1}\leq n_{2k-2}$, we see that $L'$ can be obtained by folding a tangle in the link $J'=J(n_1, \dots n_{2k-2},n_{2k-2}+n_{2k}-n_{2k-1})$. $J'$ takes the form
\[
J' =\M\left(0; \dots, \underbrace{\frac{p}{q}, \dots,\frac{p}{q}}_{n_{2k-2}+n_{2k}-n_{2k-1}} \right)
\] and one may break up the $n_{2k-2}+n_{2k}-n_{2k-1}$ tangles as
\[
J' =\M\left(0; \dots, \underbrace{\frac{p}{q}, \dots,\frac{p}{q}}_{n_{2k-2}-n_{2k-1}},\underbrace{\frac{p}{q}, \dots,\frac{p}{q}}_{n_{2k-1}},\underbrace{\frac{p}{q}, \dots,\frac{p}{q}}_{n_{n_{2k}-n_{2k-1}}} \right)
\]
We perform the folding on a disk containing the sequence of $n_{2k-1}$ rational tangles. This folding produces the link $L'$.

By induction we can assume that $J'$ is weakly doubly slice with both quasi-orientations and so, by Lemma~\ref{lem:folding}, $L'$ is also weakly doubly slice with both quasi-orientations. This completes the inductive step of the proof.
\end{proof}


\section{Seifert fibered spaces and embeddings}\label{sec:SFS_embedding}
In this section we prove Theorem~\ref{thm:embedding}. First we establish some notation and recall some facts concerning Seifert fibered spaces. See \cite{Neumann1978seifert} for a more in depth treatment on Seifert fibered spaces and plumbings. In this paper we will use $S^2(e; \frac{p_1}{q_1}, \frac{p_2}{q_2}, \dots, \frac{p_k}{q_k})$ to denote the Seifert fibered space obtained by the surgery on the diagram given in Figure~\ref{fig:sfs_as_surgery}, where $e\in\Z$ and $p_i/q_i \in \Q$, where we assume that $p_i$ and $q_i$ are a pair of coprime integers with $|p_i|>1$ for all $i$. Give such a presentation the generalized Euler invariant of $Y$, denoted $\varepsilon(Y)$, can be computed as:
\[
\varepsilon(Y) = e - \sum_{i=1}^k \frac{q_k}{p_k}.
\]
Notice that a surgery description as in Figure~\ref{fig:sfs_as_surgery} is far from unique. In particular, one can perform Rolfsen twists on the $p_i/q_i$ framed components to 
show that
\[
Y\cong S^2\left(e; \frac{p_1}{q_1}, \frac{p_2}{q_2}, \dots, \frac{p_k}{q_k}\right)
\]
and 
\[
Y'\cong S^2\left(e'; \frac{p_1'}{q_1'}, \frac{p_2'}{q_2'}, \dots, \frac{p_k'}{q_k'}\right)
\]
are homeomorphic if
\[
\varepsilon(Y)=\varepsilon(Y')
\]
and there is some permutation $\sigma$ of $\{1, \dots, k \}$ such that for all $i$
\[
\frac{q_i}{p_i}\equiv \frac{q'_{\sigma(i)}}{p'_{\sigma(i)}} \bmod{1}.
\]

One can compute the homology of a Seifert fibered space from this surgery description (cf. \cite[Lemma~4.1]{Issa2018embedding}).
\begin{lem}\label{lemma:sfs_homology}
Let $Y$ be the Seifert fibered space $Y\cong S^2(e; \frac{p_1}{q_1}, \dots, \frac{p_k}{q_k})$.  For $j$ in the range $1\leq j\leq k$ define
\[
d_j=
\begin{cases}
1 &\text{if $j=1,2$}\\
\gcd\{ p_{\sigma(1)}p_{\sigma(2)}\cdots p_{\sigma(j-2)} \mid \sigma\in S_{k} \} &\text{if $3\leq j\leq k$.}
\end{cases}
\] and set
\[
d_{k+1}=(p_{1} \cdots p_k)\varepsilon(Y).
\]
Then $Y$ has homology
\[H_1(Y;\Z)\cong \bigoplus_{i=1}^{k} \frac{\Z}{D_i\Z},\]
where $D_i=d_{i+1}/d_i$. \qed
\end{lem}
In particular, this implies that 
\begin{equation}\label{eq:SFb1_calc}
b_1(Y)= \begin{cases}
0 &\text{if $\varepsilon(Y)\neq 0$}\\
1 &\text{if $\varepsilon(Y)=0$}
\end{cases}
\end{equation}

If $Y$ satisfies $\varepsilon(Y)\geq 0$, then we can normalize $Y$ so that it takes the form
\[
Y\cong S^2\left(e; \frac{p_1}{q_1}, \frac{p_2}{q_2}, \dots, \frac{p_k}{q_k}\right)
\]
where $\frac{p_i}{q_i}>1$ for all $i$ and $e>0$. Given the standard presentation for $Y$ we can construct a positive semi-definite plumbing bounding $Y$.

Given a rational number $r > 1$, there is a unique (negative) continued fraction expansion
\[
r = [a_1, \ldots, a_{n}]^- := a_1 - \cfrac{1}{a_2 - \cfrac{1}{\begin{aligned}\ddots \,\,\, & \\[-3ex] & a_{n-1} - \cfrac{1}{a_{n}} \end{aligned}}},
\]
where $n \geq 1$ and $a_i \geq 2$ are integers for all $i \in \{1, \ldots, n\}$. We associate to $r$ the weighted linear graph (or linear chain) given in Figure~\ref{fig:linearchain}. We call the vertex with weight labelled by $a_i$ the $i$th vertex of the linear chain associated to $r$, so that the vertex labelled with weight $a_1$ is the first, or starting vertex of the linear chain.

\begin{figure}[ht]
  \begin{overpic}[width=150pt]{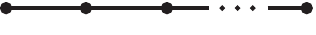}
    \put (0, 0) {$a_1$}
    \put (24, 0) {$a_2$}
    \put (50, 0) {$a_3$}
    \put (95, 0) {$a_n$}
  \end{overpic}
  \caption{Weighted linear chain representing $r = [a_1,\ldots,a_n]^-$.}
  \label{fig:linearchain}
\end{figure}

For each $i\in\{1,\ldots, k\}$, we have the unique continued fraction expansion $\frac{p_i}{q_i} = [a_1^i, \ldots, a_{h_i}^i]^-$ where $h_i \ge 1$ and $a_j^i \ge 2$ are integers for all $j \in \{1,\ldots,h_i\}$. We associate to $Y = S^2(e; \frac{p_1}{q_1}, \ldots, \frac{p_k}{q_k})$ the weighted star-shaped graph in Figure~\ref{fig:plumbing}. The $i$th leg (sometimes also called the $i$th arm) of the star-shaped graph is the weighted linear subgraph for $p_i/q_i$ generated by the vertices labelled with weights $a_1^i, \dots, a_{h_i}^i$. The degree $k$ vertex labelled with weight $e$ is called the central vertex.

\begin{figure}[!ht]
  \begin{overpic}[height=150pt]{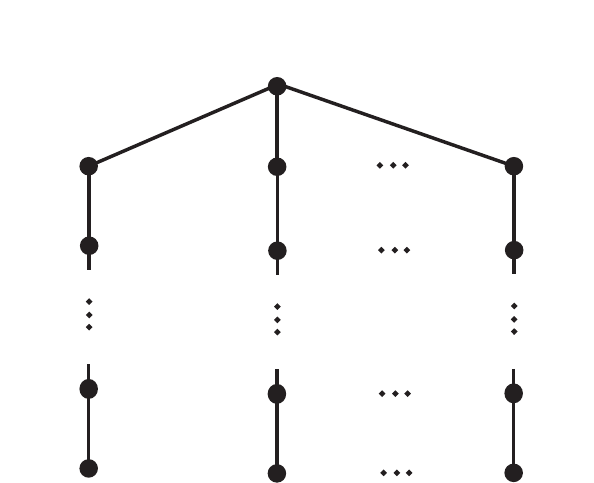}
    \put (45, 73) {\Large $e$}
    \put (4, 52) {\large $a_1^1$}
    \put (4, 38) {\large $a_2^1$}
    \put (2, 2) {\large $a_{h_1}^1$}
    
    \put (36, 52) {\large $a_1^2$}
    \put (36, 38) {\large $a_2^2$}
    \put (34, 2) {\large $a_{h_2}^2$}
    
    \put (90, 52) {\large $a_1^k$}
    \put (90, 38) {\large $a_2^k$}
    \put (90, 2) {\large $a_{h_p}^k$}    


  \end{overpic}
  \caption{The weighted star-shaped plumbing graph $\Gamma$.}
  \label{fig:plumbing}
\end{figure}

Let $\Gamma$ be this weighted star-shaped graph for $Y$ and let $X_\Gamma$ be the oriented smooth $4$-manifold obtained by plumbing $D^2$-bundles over $S^2$ according to the weighted graph $\Gamma$.  Denote vertices of $\Gamma$ by $v_1, v_2, \ldots, v_m$. The zero-sections of the $D^2$-bundles over $S^2$ corresponding to each of $v_1, \ldots, v_m$ in the plumbing together form a natural spherical basis for $H_2(X_\Gamma)$. With respect to this basis, which we call the vertex basis, the intersection form of $X_\Gamma$ is given by the weighted adjacency matrix $Q_\Gamma$ with entries $Q_{ij}$, $1 \le i,j \le m$ given by
\[
Q_{ij} = \begin{cases} 
      \text{w}(v_i), & \mbox{if }i = j \\
      -1, & \mbox{if }v_i\mbox{ and }v_j\mbox{ are connected by an edge} \\
      0, & \mbox{otherwise} 
   \end{cases},
\]
where $\text{w}(v_i)$ is the weight of vertex $v_i$. Denoting by $Q_X$ the intersection form of $X$, we call $(H_2(X), Q_X) \cong (\Z^m, Q_\Gamma)$ the intersection lattice of $X_{\Gamma}$ (or of $\Gamma$). 

\begin{lem}\label{lem:kernel_properties}
If $\varepsilon(Y)=0$, then $X_\Gamma$ is positive semi-definite with nullity of rank one. Let $L=\lcm(p_1, \dots, p_k)$. Then is a row vector $v_0$ with integer entries such that $v_0 Q=0$ and
\begin{enumerate}[(i)]
    \item The coefficient of $v_0$ on the central vertex is $L$
    \item For a coefficient corresponding to a leaf of the $i$th arm is $\frac{L}{p_i}$. 
\end{enumerate}
\end{lem}
\begin{proof}
Consider first the plumbing obtained by deleting the central vertex. This consists of a disjoint union of linear chains in each each vertex has weight at least two. This is easily seen to be positive definite. Thus we see that $X_\Gamma$ has a positive definite subspace of codimension one. Thus to to establish that it is positive semi-definite with nullity of rank one, it suffices to exhibit the vector $v_0$ as described in the statement of the lemma. The remainder of this proof is taken up with constructing this $v_0$.

First let $p/q=[a_1, \dots, a_n]^{-}$ be a continued fraction. Define the integers $b_1, \dots, b_{n+1}$ recursively by the conditions that $b_{n+1}=0$, $b_{n}=1$ and $b_{k-1}= a_{k} b_k - b_{k+1}$ for $1<k\leq n$. For integers defined this way we have that
\[
(0, b_1, \dots, b_n)\begin{pmatrix}
k & -1    &     & \\
-1 & a_1 & -1 &\\
    & -1   & \ddots  & -1\\
    &       & -1 & a_n
\end{pmatrix}
=
(-b_1, a_1 b_1 -b_2,0, \dots, 0).
\]
However notice that the $b_k$ satisfy the same recursion relation as the denominators of the sequence of continued fractions $[a_k, \dots, a_n]$ as $k$ decreases. Thus we see that the $b_k$ are precisely these denominators. and we have $b_1=q$ and $p=a_1 b_1 -b_2$.

Thus if we now consider the full plumbing matrix $Q$ with rows ordered so that the central vertex corresponds to the first column. The construction of the previous paragraph shows that for the $i$th arm we have vector $v_i$, which has non-zero entries only on the entries corresponding to vertices of the $i$th arm and take value $1$ on the leaf of the arm such that
$v_i Q =(-q_i, 0,\dots,0, p_i, 0, \dots,0 )$, where the $p_i$ occurs for the vertex of the arm adjacent to the central vertex. 

We define the vector $v_0$ to be the linear combination 
\[
v_0= Le_1 + \sum_{j=1}^k \frac{L}{p_j} v_j,
\]
where $e_1=(1,0, \dots, 0)$ and $L=\lcm(v_1, \dots v_k)$. It is a calculation using the fact that $\varepsilon(Y)=0$ to show that this satisfies $v_0 Q =0$. The other properties are evident from the construction.
\end{proof}


We will make use of the classification of Seifert fibered spaces which bound smooth $\Q H_*(S^1\times B^3)$s. The classification was first proven by Aceto \cite{Aceto2015plumbed}, although one implication was implicit in the work of Donald \cite{Donald2015Embedding}. As was demonstrated in \cite{Issa2018bounding_definite}, this result can also be deduced relatively easily from Theorem~\ref{thm:emb_ineq_eq_case} below.
\begin{thm}\label{thm:boundingS1timesB3}
Let $Y$ be a Seifert fibered space over $S^2$. Then $Y$ is the boundary of a smooth $\Q H_*(S^1\times B^3)$ if and only if $Y$ is homeomorphic to a space of the form
\[
Y \cong S^2\left( 0;  \frac{p_1}{q_1}, - \frac{p_1}{q_1}, \dots,  \frac{p_k}{q_k}, - \frac{p_k}{q_k}  \right),
\]
where $p_i> q_i \geq 1$ are coprime integers for all $i$.
\end{thm}
\begin{rem}\label{rem:alternateforms} We make two comments on the form of the space $Y$ appearing in Theorem~\ref{thm:boundingS1timesB3}.
\begin{enumerate}[(i)]
    \item\label{it:pair_exchange} Whenever we have a pair of invariants $\frac{p_i}{q_i}$ and $-\frac{p_i}{q_i}$, we can replace these with $-\frac{p_i}{p_i-q_i}$ and $\frac{p_i}{p_i-q_i}$ without changing the homeomorphism type of $Y$, thus we can further assume that the coefficients rational numbers $\frac{p_i}{q_i}$ in Theorem~\ref{thm:boundingS1timesB3} satisfy $\frac{p_i}{q_i}\geq 2$.
    \item\label{it:standard_form} The standard presentation for $Y$  in Theorem~\ref{thm:boundingS1timesB3} is
\[
Y \cong S^2\left( k;  \frac{p_1}{q_1},  \frac{p_1}{p_1-q_1}, \dots,  \frac{p_k}{q_k},  \frac{p_k}{p_k-q_k}  \right)
\]
where $p_i/q_i>1$ for all $i$.
\end{enumerate}
\end{rem}
The following theorem on lattice embeddings will be useful.
\begin{thm}[Theorem 6 of \cite{Issa2018bounding_definite}] \label{thm:emb_ineq_eq_case}
Let $\iota : (\Z^{|\Gamma'|}, Q_{\Gamma'}) \rightarrow (\Z^m, \mbox{Id})$ be a lattice embedding, where $m > 0$ and $\Gamma'$ is a disjoint union of weighted linear chains representing fractions $\frac{p_1}{q_1}, \ldots, \frac{p_n}{q_n} \in \Q_{>1}$. Suppose that there is a unit vector $w \in (\Z^m, \mbox{Id})$ which pairs non-trivially with (the image of) the starting vertex of each linear chain. Then
 \[
 \sum_{i=1}^n \frac{q_i}{p_i} \leq 1.
 \]
\end{thm}
Our key application of Theorem~\ref{thm:emb_ineq_eq_case} is through the following lemma. 
\begin{lem}\label{lem:matrix_factorization}
Let $Y$ be a Seifert fibered space over $S^2$ in the form 
\[Y= S^2\left( \ell;  \frac{p_1}{q_1}, \dots, \frac{p_{2\ell}}{q_{2\ell}}  \right)\] 
with $\varepsilon(Y)=0$.
Suppose the matrix $Q$ admits as factorization of the form $Q= A^T A$, where $A$ is an integer matrix.
Then up to reordering columns and multiplying the columns by -1, the row corresponding to
the central vertex takes the form
\[
(\underbrace{1, \dots, 1}_{\ell}, 0, \dots,0).
\]
For each $i=1, \dots, \ell$, let $C_i\subseteq \{1, \dots, 2 \ell\}$ be the set of indices such that $j\in C_i$ if and only if a vertex on the $j$th arm has non-zero entry in its $i$th column. Then the $C_i$ form a partition of the $\{1,\dots, 2\ell\}$ into classes such that each class contains two elements and for any $i$ we have that if $C_i=\{\alpha, \beta\}$, then $\frac{q_\alpha}{p_\alpha}+\frac{q_\beta}{p_\beta}=1$.
\end{lem}
\begin{proof}
By rearranging the columns and multiplying the columns by $-1$ as necessary we can assume that the row corresponding to the central vertex takes the form
\[
(c_1, \dots, c_{\ell'}, 0 ,\dots, 0)
\]
where the $c_i\geq 1$ and $\ell' \leq \ell$. Now for each $i=1, \dots, \ell'$, let $C_i'$ be the set of arms for which the lead vertex has a non-zero coefficient in the $i$th column. Now since the leading vertex of each arm pairs non-trivially with the central vertex, every arm appears in at least one $C_i'$. Now Theorem~\ref{thm:emb_ineq_eq_case} implies that for each $C_i'$ we have
\[
\sum_{j \in C_i'} \frac{q_j}{p_j} \leq 1.
\]
However as the $\{1, \dots, 2\ell\}= \cup_{i=1}^{\ell'} C_i'$ we have that
\[
\ell=\sum_{i=1}^{2\ell} \frac{q_i}{p_i}\leq \sum_{i=1}^{\ell'} \sum_{j \in C_i'} \frac{q_j}{p_j}  \leq \ell' \leq \ell.
\]
Thus all these inequalities must in fact be equalities. This implies that $\ell'=\ell$, which in turn implies that $c_1=\dots= c_\ell =1$. This implies that the row corresponding to the central vertex is in the required form. Secondly it implies that the $C_i'$ are in fact a partition, in particular, every $j$ appears in precisely one of the $C_i$. Finally, it implies that for each of the $C_i'$ we have
\begin{equation}\label{eq:partition_equality}
\sum_{j \in C_i'} \frac{q_i}{p_i} = 1.
\end{equation}
However the $C_i'$ are partitioning $2\ell$ arms amongst $\ell$ classes. For \eqref{eq:partition_equality} to hold we see that each class must contain at least two elements. Thus each $C_i'$ contains precisely two elements and their reciprocals sum to one. 

Notice that by construction we have $C_i' \subseteq C_i$; the $C_i'$ were constructed by considering the arms for which the leading vertices had a non-zero coefficient in the $i$th column, whereas the $C_i$ are the arms with some vertex with a non-zero coefficient in the $i$th column. We finish the proof by showing that $C_i = C_i'$ for all $i$. 

Suppose otherwise, then we would we have a vertex $v$ on the $j$th arm where $j\not\in C_i'$ such that the $i$th column of the row corresponding to $v$ is non-zero. 
However, we can view the vertex $v$ as a linear chain in its own right. So if $C_i' = \{\alpha, \beta\}$, then Theorem~\ref{thm:emb_ineq_eq_case} would imply that
\[\frac{q_\alpha}{p_\alpha}+\frac{q_\beta}{p_\beta}+ \frac{1}{w(v)}= 1+ \frac{1}{w(v)}\leq 1,
\]
which is impossible.
\end{proof}
\begin{figure}[!ht]
\vspace{10pt}
  \begin{overpic}[height=100pt]{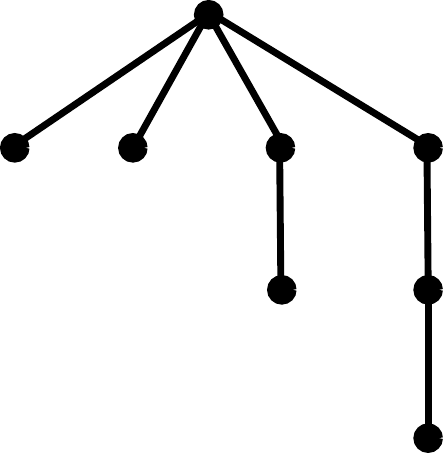}
    \put (45, 102) {$2$}
    \put (-8, 62) {$2$}
    \put (18, 62) {$2$}
    \put (50, 32) {$3$}
    \put (50, 62) {$3$}
    \put (99, 0) {$2$}
    \put (99, 32) {$3$}
    \put (99, 62) {$2$}
  \end{overpic}
  \caption{The semi-definite plumbing for $S^2 \left(2; \frac{2}{1}, \frac{2}{1}, \frac{8}{3}, \frac{8}{5} \right)$}
  \label{fig:example_plumbing}
\end{figure}
\begin{exam}\label{ex:main}
Consider the Seifert-fibered space $S^2 \left(2; \frac{2}{1}, \frac{2}{1}, \frac{8}{3}, \frac{8}{5} \right)$. This arise as the boundary of the positive semi-definite plumbing displayed in Figure~\ref{fig:example_plumbing}. The intersection form is represented by the matrix:
\[
Q= \begin{pmatrix}
2 &-1&-1& -1&  & -1&  &  \\
-1&2 &  &   &  &   &  &  \\
-1&  & 2&   &  &   &  &  \\
-1&  &  & 3 &-1&   &  &  \\
  &  &  &-1 & 3&   &  &   \\ 
-1&  &  &   &  & 2 &-1&   \\
  &  &  &   &  &-1 & 3&-1 \\
  &  &  &   &  &   &-1& 2 \\
\end{pmatrix}
\]
The vector $v_0$ satisfying $v_0 Q =0$ constructed in Lemma~\ref{lem:kernel_properties} takes the form
\[
v_0=(8, 4, 4, 3, 1 , 5 ,2, 1).
\]

In this case, it turns out that there is an essentially unique candidate for the matrix $A^T$ satisfying $A^T A =Q$:
\[
A^T= \begin{pmatrix}
1 &1 &  &   &  &  &    \\
-1&  &1 &   &  &   &   \\
-1&  &-1&   &  &   &    \\
  &-1&  & 1 &1 &   &    \\
  &  &  &   &-1&1  &1    \\ 
  &-1&  &-1 &  &   &     \\
  &  &  &1  &-1 &-1&    \\
  &  &  &   &  &  1&-1 \\
\end{pmatrix}.
\]
Notice that the central vertex corresponds to the first row in $A^T$ and consistent with Lemma~\ref{lem:matrix_factorization}, the arms with non-zero entries in the first column correspond to the fractions $2/1$ and $2/1$ and the arms with non-zero entries in the second row correspond to $8/3$ and $8/5$.
\end{exam}

\subsection{Homological properties of embeddings}\label{sec:homological_material}
Next we recall some homological results concerning embeddings of manifolds.
\begin{lem}\label{lem:splitting_principle}
Suppose that $Y$ is a connected oriented $3$-manifold that smoothly embeds into a connected oriented $4$-manifold $Z$. If the induced map $H_1(Y;\Z)\rightarrow H_1(Z;\Z)$ is surjective, then $Y$ separates $Z$ into two components.
\end{lem}
\begin{proof}
Since $Y$ and $Z$ are orientable the normal bundle $\nu Y$ of $Y$ in $Z$ is trivial. If $Y$ does not separate $Z$, then we can find an arc in $Z \setminus Y$ that connects the two components of $\nu Y \setminus Y$. This gives a closed curve $\gamma$ in $Z$ that intersects $Y$ transversely in a single point. Thus the homology class represented by $\gamma$ has non-trivial pairing with the class $[Y]\in H_3(Z;\Z)$. Hence we see that $\gamma$ cannot be in the image of $H_1(Y;\Z)$ in $H_1(Z;\Z)$.
\end{proof}

Next we study the effect of embedding $\Q H_*(S^1\times S^2)$s into $\Z H_*(S^1 \times S^3)$s.
\begin{lem}\label{lem:Ui_properties}
Suppose that $Y$ is a $\Q H_*(S^1\times S^2)$ which embeds into $Z$, a $\Z H_*(S^1 \times S^3)$, and the map $H_1(Y; \Z) \rightarrow H_1(Z;\Z)$ induced by inclusion is surjective. Then $Z$ can be decomposed as $Z= U_1 \cup_Y U_2$, where $U_1$ and $U_2$ are submanifolds $\partial U_1 \cong -\partial U_2 \cong Y$ with the following properties: 
\begin{enumerate}[(i)]
\item\label{enum:H1Z_surj} inclusion induces a surjection $H_1(Y;\Z) \rightarrow H_1(U_i;\Z)$;
\item\label{enum:H1Q_iso} inclusion induces an isomorphism $H_1(Y;\Q) \rightarrow H_1(U_i;\Q)$;
\item\label{enum:H3Z_zero} $H_3(U_i; \Z)=0$;
\item\label{enum:H2Z_zero} $H_2(U_i; \Z)=0$;
\item\label{enum:H2tors_iso} The map 
\[H^2(U_1;\Z) \oplus H^2(U_2;\Z) \rightarrow \tor H^2(Y;\Z)\]
induced by inclusion is an isomorphism.
\end{enumerate}
In particular the $U_i$ are both  $\Q H_*(S^1\times B^3)$s. 
\end{lem}
\begin{proof}
By Lemma~\ref{lem:splitting_principle}, $Y$ separates $Z$ into two components with the required boundary components. We establish the necessary homological properties.

We have the exact sequence from Mayer-Vietoris:
\begin{equation}\label{eq:MV_H_1}
0\rightarrow H_1(Y;\Z) \rightarrow H_1(U_1;\Z)\oplus H_1(U_2;\Z) \rightarrow H_1(Z;\Z) \rightarrow 0.
\end{equation}
Since $H_1(Z; \Z)$ is torsion free, we see that $\tor H_1(U_1;\Z)\oplus \tor H_1(U_2;\Z)$ is contained in the image of $H_1(Y;\Z)$.

Since $b_1(Y)=b_1(Z)=1$, we see that $b_1(U_1)+b_1(U_2)=2$. However the surjection $H_1(Y;\Z)\rightarrow H_1(Z;\Z)$ factors as
\[
H_1(Y;\Z)\rightarrow H_1(U_i;\Z) \rightarrow H_1(Z;\Z).
\]
This shows that $b_1(U_1)=b_1(U_2)=1$ and hence that $H_1(Y;\Z)\rightarrow H_1(U_i;\Z)$ is a surjection. This establishes \eqref{enum:H1Z_surj}. Statement \eqref{enum:H1Q_iso} is established similarly, but working with $\Q$ coefficients instead.

Using \eqref{enum:H1Z_surj}, the long exact sequence of the pair $(U_i,Y)$ shows that $H_1(U_i, Y;\Z)=0$. By Lefschetz duality this implies that $H^3(U_i;\Z)=0$. Since $H_3(U_i;\Z)$ is torsion free, this implies that $H_3(U_i;\Z)=0$. By Universal coefficients we see also that $H_2(U_i; \Z)$ is torsion free.


Now consider the following piece of the Mayer-Vietoris sequence
\[
0 \rightarrow H_3(Z;\Z) \rightarrow H_2(Y;\Z) \rightarrow H_2(U_1;\Z)\oplus H_2(U_2;\Z) \rightarrow 0.
\]
Since $H_3(Z;\Z)$ and $H_2(Y;\Z)$ are both isomorphic to $\Z$. This implies that $H_2(U_1;\Z)\oplus H_2(U_2;\Z)=0$, which establishes \eqref{enum:H2Z_zero}.

Finally, we establish \eqref{enum:H2tors_iso}. To do this, look at the following portion of the Mayer-Vietoris sequence for cohomology:
\begin{equation}\label{eq:MVH^2}
0 \rightarrow H^2(U_1;\Z)\oplus H^2(U_2;\Z) \rightarrow H^2(Y;\Z) \rightarrow H^3(Z;\Z).
\end{equation}
Since $H_2(U_i;\Z)=0$, the term $H^2(U_1;\Z)\oplus H^2(U_2;\Z)$ consists entirely of torsion. Thus the image of $H^2(U_1;\Z)\oplus H^2(U_2;\Z)$ in \eqref{eq:MVH^2} is entirely contained in $\tor H^2(Y;\Z)$. On the other hand $H^3(Z;\Z)\cong \Z$ is torsion-free, so $H^2(U_1;\Z)\oplus H^2(U_2;\Z)$ must surject onto $\tor H^2(Y;\Z)$.
\end{proof}

Finally, we recall the conditions for two 4-manifolds to glue to give a closed definite 4-manifold \cite[Proposition 7]{Issa2018bounding_definite}.

\begin{prop}\label{prop:def_gluing_thm}
Let $X_1$ and $X_2$ be 4-manifolds with $\partial X_1 = -\partial X_2 =Y$. Then the closed 4-manifold $Z=X_1 \cup_Y X_2$ is positive definite if and only if
\begin{enumerate}[(a)]
\item\label{enum:injcond} the inclusion-induced map 
\[
(i_1)_* \oplus (i_2)_* \colon H_1(Y; \Q) \rightarrow H_1(X_1; \Q)\oplus H_1(X_2; \Q)
\]
is injective and
\item\label{enum:sigeq} both $X_1$ and $X_2$ are positive semi-definite.
\end{enumerate}\qed
\end{prop}

\subsection{Obstructing embeddings}\label{sec:embedding_thm}
For the duration of this section, we will take $Y$ to be a Seifert fibered space over $S^2$ that embeds smoothly into some $\Z H_*(S^1\times S^3)$, which we will call $Z$, and assume that the induced map $H_1(Y;\Z) \rightarrow H_1(Z;\Z)$ is surjective. The objective is to show that $Y$ is in the form required by Theorem~\ref{thm:embedding}\eqref{it:explicit_Description} and hence establish
the implication \eqref{it:embedding} $\Rightarrow$ \eqref{it:explicit_Description} in Theorem~\ref{thm:embedding}.
By Lemma~\ref{lem:Ui_properties}, the embedding of $Y$ into $Z$ decomposes $Z$ as $Z=U_1 \cup_Y -U_2$, where $U_1, U_2$ are $\Q H_*(S^1\times B^3)$s with $\partial U_1 \cong \partial U_2 \cong -Y$. 

By Theorem~\ref{thm:boundingS1timesB3}, the existence of $U_1$ implies that $Y$ takes the form:
\begin{equation}\label{eq:starting_form}
Y\cong S^2 \left(0; \left\{\frac{p_1}{q_1}, -\frac{p_1}{q_1}\right\}^{\times n_1}, \dots, \left\{\frac{p_k}{q_k}, -\frac{p_k}{q_k}\right\}^{\times n_k}  \right),
\end{equation}
where, for each $i$, we assume that $p_i$ and $q_i$ are coprime integers with $\frac{p_i}{q_i}\geq 2$ (cf. Remark~\ref{rem:alternateforms}\eqref{it:pair_exchange}); $n_i\geq 1$ for all $i$; and $p_i/q_i \neq p_j/q_j,$ whenever $i\neq j$. Thus, in order to establish that $Y$ is in form required by Theorem~\ref{thm:embedding}\eqref{it:explicit_Description} we need to show that $\gcd(p_i, p_j)=1$ whenever $i\neq j$.

By applying Rolfsen twists to the presentation in \eqref{eq:starting_form}, as discussed at the start of Section~\ref{sec:SFS_embedding}, we see that $Y$ can be written in the form
\begin{equation}\label{eq:pos_def_form}
Y\cong S^2 \left(e; \left\{\frac{p_1}{q_1}, \frac{p_1}{p_1-q_1}\right\}^{\times n_1}, \dots, \left\{\frac{p_k}{q_k}, \frac{p_k}{p_k-q_k}\right\}^{\times n_k}  \right),
\end{equation}
where $e=\sum_{i=1}^k n_i$ and $\left\{\frac{p_i}{q_i}, \frac{p_i}{p_i-q_i}\right\} \neq  \left\{\frac{p_j}{q_j}, \frac{p_j}{p_j-q_j}\right\}$ whenever $i\neq j$.

Using the presentation in \eqref{eq:pos_def_form} we see that $Y$ is the boundary of positive semi-definite plumbing $X$ intersection form $(\Z^n, Q_{\Gamma})$, where $\Gamma$ is a weighted star-shaped graph as in Figure~\ref{fig:plumbing} and $n = |\Gamma|$ is the number of vertices in $\Gamma$. For $i=1,2$ we form the closed smooth manifolds
\[
W_i=X \cup U_i.
\]
\begin{lem}\label{lem:Wi_properties}
For $i=1,2$, $W_i$ has a positive definite intersection form and $H_2(W_i;\Z)$ is torsion-free and of rank $n-1$.
\end{lem}
\begin{proof}
The $W_i$ are definite by Theorem~\ref{prop:def_gluing_thm}: by Lemma~\ref{lem:Ui_properties}\eqref{enum:H1Q_iso} the map 
\[H_1(Y;\Q)\rightarrow H_1(U_i;\Q)\]
is an isomorphism and both $X$ and $U_i$ are positive semi-definite.

Using the Mayer-Vietoris sequence for homology shows that $H_1(W_i;\Z)$ sits inside the exact sequence
\begin{equation*}
 H_1(Y;\Z) \rightarrow H_1(U_i;\Z) \oplus H_1(X;\Z) \rightarrow H_1(W_i;\Z) \rightarrow 0.
\end{equation*}
Since $H_1(X;\Z)=0$ and, by Lemma~\ref{lem:Ui_properties}\eqref{enum:H1Z_surj} the map $H_1(Y;\Z) \rightarrow H_1(U_i;\Z)$ is surjective, we see that $H_1(W_i;\Z)=0$. By universal coefficients, this implies that $H^2(W_i; \Z)$ is torsion-free. Poincar\'e duality implies that $H_2(W_i; \Z)$ is torsion-free too.

The manifold $W_i$ is definite, and by Novikov additivity its signature is $n-1$, i.e the signature of $X$. Thus the rank of $H_2(W_i;\Z)$ must be $n-1$. 
\end{proof}

Let us now fix bases for some of the homology groups in use. These will be in use for the remainder of the section.

Since $W_i$ is smooth, closed and definite its intersection form on $H_2(W_i;\Z)$ is diagonalizable \cite{Donaldson1987orientation}. Thus we can can select an orthonormal basis $e_1, \dots, e_{n-1}$ for $H_2(W_i;\Z)$. For $H_2(X;\Z)$ we fix the basis corresponding to vertices of the plumbing graph $\Gamma$. Thus with respect to this basis the intersection form of $X$ is presented by $Q= Q_\Gamma$.

Consider the map $\iota_i :H_2(X;\Z) \rightarrow H_2(W_i;\Z)$ induced by inclusion. Suppose that it is represented by a matrix $A_i$ (of size $(n-1)\times n$) with respect to our chosen bases. That is, if we represent elements of the homology groups by column vectors with respect to these bases, then $\iota(v)=A_i v$.\footnote{The columns of $A_i$ are the images of vertices.} Since $\iota_i$ preserves the intersection form we have that
\[
w^T Q v = (A_i w)^T A_i v = w^T A_i^T A_i v
\]
for all $v,w \in H_2(X)$ and thus we see that each $A_i$ satisfies 
\begin{equation}\label{eq:Ai_factorization}
A_i^T A_i=Q.
\end{equation}
Since $H_2(W_i;\Z)$ is torsion-free, universal coefficients shows that we have an isomorphism $H^2(W_i;\Z) \cong \Hom(H_2(W_i;\Z),\Z)$. This allows us to take the basis for $H^2(W_i;\Z)$ which is algebraically dual to our chosen basis for $H_2(W_i;\Z)$. That is, we take basis the $e_1^*, \dots, e_{n-1}^*$ satisfying
\[
e_i^*(e_j)= \delta_{ij}.
\]
Similarly, as $H_2(X;\Z)$ is torsion-free, we have an isomorphism $H^2(X;\Z) \cong \Hom(H_2(X;Z),\Z)$, allowing us to take the basis for $H^2(X;\Z)$ which is algebraically dual to the vertex basis of $H_2(X;\Z)$. 

Now the map 
\[
H^2(W_i;\Z) \rightarrow H^2(X;\Z)
\]
induced by inclusion is dual to the map $H_2(X;\Z) \rightarrow H_2(W_i;\Z)$. Thus with respect to the dual bases, the map $H^2(W_i;\Z) \rightarrow H^2(X;\Z)$ is represented by the matrix $A_i^T$.

The technical heart of Theorem~\ref{thm:embedding} is the following lemma which allows to further understand the $A_i^T$. Recall that the matrix $Q$ has kernel of dimension one (see Lemma~\ref{lem:kernel_properties}).
\begin{lem}\label{lem:tors_condition}
Let $v_0$ be a non-zero vector such that $v_0^T Q=0$. Then
\begin{equation}\label{eq:tor_condition}
\im (A_1^T | A_2^T) = \{x\in \Z^n \,|\, v_0^T x =0 \}.
\end{equation}
\end{lem}
\begin{proof}
First fix $i\in\{1,2\}$. By considering the long exact sequences of pairs and the inclusion $(X, Y) \xhookrightarrow{} (W_i, U_i)$, we have the following commutative diagram with exact rows:
\[\xymatrix{
  H^2(W_i,U_i) \ar[d]^{\cong i_1} \ar[r]^\alpha &H^2(W_i)\ar[d]^{i_2}\ar[r]^\beta & H^2(U_i) \ar[d]^{i_3} \ar[r] & 0\\
  H^2(X,Y) \ar[r]^\gamma   &H^2(X)\ar[r]^\delta & H^2(Y) \ar[r] & 0.}\]
 The zeroes on the right hand side come from the fact that $H^3(X,Y)\cong H_1(X)=0$ and that $H^3(X,Y)\cong H^3(W_i, U_i)$ by excision. Excision shows that $i_1$ is an isomorphism.
 
Since $H^2(X,Y)$ is isomorphic to $H_2(X)$ by Poincar\'e duality, we can choose a basis for $H^2(X,Y)$ with respect to which the map $\gamma$ is represented by the matrix $Q$.

Since the map $\delta$ is surjective. This allows us to identify $H^2(Y)$ with the cokernel of $\gamma$, in particular this allows us to identify $H^2(Y)$ with cokernel of $Q$. Since $\beta$ is surjective we have that $\im (i_3 \circ \beta) =\im i_3$. Thus the image of $i_3$ is identified with the image of $i_2$ in the cokernel of $\gamma$. Since the map $i_2$ is represented by $A_i^T$ with respect to the coordinates in use this shows we can identify $\im i_3$ with $\im A_i^T / \im Q$.

By Lemma~\ref{lem:Ui_properties} the map $H^2(U_1) \oplus H^2(U_2) \rightarrow \tor H^2(Y)$ is an isomorphism. It follows that the image of the matrix $(A_1^T | A_2^T)$ is precisely the set of elements in $\Z^n$, which represent torsion modulo $Q$. We identify this set now.

 Let $v_0$ be a non-zero vector such that $v_0^T Q=0$. Notice that $x\in \Z^n$ represents torsion in $\Z^n/Q\Z^n$ if and only if $mx \in Q \Z^n$ for some integer $m$. Equivalently if and only if there is $w \in \Q^n$ such that $x=Qw$. Since $v_0^T Q=0$, this implies that $v_0^T x =0$. Conversely, we have have that $Q \Q^n$ is an $(n-1)$-dimensional subspace of $\Q^n$. So so we have that 
 \[Q \Q^n = \{w\in \Q^n \,|\, v_0^T w =0\}.\]
 Thus we see that the set of $v\in \Z^n$ which represent torsion modulo $Q$ is precisely 
 \[\{x\in \Z^n \,|\, v_0^T x =0 \},\]
which establishes \eqref{eq:tor_condition}.
\end{proof}

Next we need to better understand the structure of the matrices $A_i^T$. Note that the rows of $A_i^T$ correspond to vertices of the plumbing graph $\Gamma$. We will assume that these rows are ordered so that the central vertex of the plumbing is the first row.
\begin{prop}\label{prop:Ai_structure}
Up to rearranging columns and multiplying columns by $-1$ and choosing an appropriate ordering in the vertices the matrices $A_i^T$ take the form

\[
A_i^T =
\left(
\begin{array}{c| c| c| c}
\RowMatrix{1} & \dots & \RowMatrix{1} & \RowMatrix{0} \\ \hline
J_1^{(i)} &  & 0 & B_1^{(i)}\\ 
\vdots & \ddots & \vdots & \vdots \\
0 &  & J_k^{(i)} & B_k^{(i)}\\ 
\end{array}\right)
\]
where the rows containing the blocks $J_j^{(i)}, B_j^{(i)}$ contain the precisely vertices of the $n_j$ arms corresponding to the fractions of the form $p_j/q_i$ and $p_j/(p_j -q_j)$.
\end{prop}
\begin{proof}
This is an application of Lemma~\ref{lem:matrix_factorization} to the matrices at hand.
\end{proof}

The following lemma will be useful for detecting whether an element is in the image of the matrix $(A_1^T | A_2^T)$. Let $\Delta_1$ be the subtree of $\Gamma$ which consists of the central vertex and the $n_1$ arms corresponding to the rational numbers of the form $\frac{p_1}{q_1}$ and $\frac{p_1}{p_1-q_1}$.
\begin{lem}\label{lem:image_detection}
There is a vector $\overline{w}$ such that
\begin{enumerate}
    \item $\overline{w}^T (A_1^T | A_2^T) \equiv 0 \bmod p_1$
    \item $\overline{w}$ has non-zero entries only the rows which correspond to vertices of $\Delta_1$.
    \item $\overline{w}$ has an entry 1 in any row which corresponds to a leaf of $\Delta_1$
\end{enumerate}
\end{lem}
\begin{proof}
Take the subplumbing $\Delta_1$ and adjust the weight on the central vertex to be $n_1$. Call this new plumbing graph $\Delta$. I.e $\Delta$ is the semidefinite plumbing which bounds the Seifert-fibered space $S^2 \left(n_1; \left\{\frac{p_1}{q_1}, \frac{p_1}{p_1-q_1}\right\}^{\times n_1} \right)$. Let $Q_\Delta$ be the matrix corresponding to the vertex basis of this plumbing. Notice that for $i=1,2$ the block matrices
\[
\widetilde{A}_i^T=
\begin{pmatrix}
\RowMatrix{1} & \RowMatrix{0} \\
J_1^{(i)} & B_1^{(i)}
\end{pmatrix}
\]
satisfy $Q_\Delta = \widetilde{A}_i^T \widetilde{A}_i$. Now let $w$ be the vector constructed in Lemma~\ref{lem:kernel_properties} satisfying $Qw=0$. Since every numerator is equal to $p_1$ the vector $w$ has entries equal to 1 on rows corresponding to a leaf of $\Delta$ and an entry $p_1$ on the row corresponding to the central vertex. Now we have that
\[
0=w^T Q w = (\widetilde{A}_i w)^T (\widetilde{A}_i w)
\]
However the right hand term here is simply the Euclidean norm of $\widetilde{A}_i w$ so it follows that $\widetilde{A}_i w=0$. Now extend $w$ by adding zero entries to get $\overline{w}\in \Z^n$ until we can perform the multiplication $\overline{w}^T A_i^T$. That is consider the product
\begin{align*}
\overline{w}^T A_i^T &= (\overline{w}^T, 0, \dots, 0)
\left(
\begin{array}{c| c| c| c}
\RowMatrix{1} & \dots & \RowMatrix{1} & \RowMatrix{0} \\ \hline
J_1^{(i)} &  & 0 & B_1^{(i)}\\ 
\vdots & \ddots & \vdots & \vdots \\
0 &  & J_k^{(i)} & B_k^{(i)}\\ 
\end{array}\right).\\
&= \overline{w}^T 
\left(
\begin{array}{c| c| c| c}
\RowMatrix{1} & \RowMatrix{1} & \RowMatrix{0} \\ \hline
J_1^{(i)} & \RowMatrix{0} & B_1^{(i)}\\ 
\end{array}\right).\\
&=(0, \dots, 0, p_1, \dots, p_1 ,0,\dots, 0)\\
&\equiv (0, \dots, 0) \bmod{p_1}
\end{align*}

Thus the vector $\overline{w}$ has all the necessary properties.
\end{proof}
We are now ready to apply the obstruction.
\begin{lem}\label{lem:obstruction_applied}
For each $i\neq j$, we have that $\gcd(p_i,p_j)=1$.
\end{lem}
\begin{proof}
We will show that $p_1$ and $p_2$ are coprime. By relabelling the $p_i$s this is sufficient to establish the lemma. Let $v_0 \in \Z^n$ be the vector satisfying $v_0^T Q=0$ as constructed in
Lemma~\ref{lem:kernel_properties}. Let $L=\lcm (p_1, \dots, p_k)$. Then the coefficient of $v_0$ on a row corresponding to a leaf of an arm corresponding to $p_1/q_1$ is $L/p_1$ and its value on a leaf of an arm corresponding to $p_2/q_2$ is $L/p_2$. Thus if we take $g=\gcd(p_1, p_2)$, there is a vector $x\in \Z^n$ of the form
\[
x^T=(0,\dots, 0, p_1/g, 0, \dots, 0, -p_2/g,0, \dots,0)
\]
such that $v_0^T x=0$. And if we take $\overline{w}$ to the vector constructed in Lemma~\ref{lem:image_detection}, we have $\overline{w}^T x = p_1/g$. However, by Lemma~\ref{lem:tors_condition}, $x$ must be in the image of $(A_1^T | A_2^T)$, so we have $x= (A_1^T | A_2^T) u$ for some $u$. However, since $\overline{v} (A_1^T | A_2^T)\equiv 0 \bmod{p_1}$, this implies that 
\[
\overline{v} x = \frac{p_1}{g} \equiv 0 \bmod p_1.
\]
This implies that $g=1$, which is the required conclusion.
\end{proof}
\begin{exam}
Returning to the Seifert fibered space $S^2 \left(2; \frac{2}{1}, \frac{2}{1}, \frac{8}{3}, \frac{8}{5} \right)$ studied in Example~\ref{ex:main}. Since 2 and 8 are not coprime, this manifold cannot be embedded into a $\Z H_*(S^1 \times S^3)$ by a map inducing a surjection on $H_1$. The vector $\overline{w}$ constructed in Lemma~\ref{lem:image_detection} is
\[
\overline{w}=(0,1,1, 0, 0,0,0, 0).
\]
This satisfies $\overline{w} A^T\equiv 0 \bmod{2}$ for any matrix $A$ satisfying $A^T A =Q$.
If we take $x$ to be the column vector
\[
x=(0, 1, 0, 0, -4, 0, 0, 0)^T, 
\]
then we see that $x$ cannot be in the image of $(A_1 | A_2)$ for any pair of matrices satisfying $A_i^T A_i= Q$ for $i=1,2$, since $\overline{w} x = 1 \not\equiv 0 \bmod{2}$. However $x$ satisfies $v_0 x =0$.
\end{exam}
Thus we summarize the proceedings of this section as follow in the following proposition which corresponds to the implication \eqref{it:embedding} $\Rightarrow$ \eqref{it:explicit_Description} in Theorem~\ref{thm:embedding}.
\begin{prop}\label{prop:obstruction_summary}
Let $Y$ be a Seifert fibered space over $S^2$ with an embedding into $Z$ an $\Z H_* (S^1 \times S^3)$ so that the map $H_1(Y;\Z) \rightarrow H_1(Z; \Z)$ is a surjection. Then $Y$ is homeomorphic to a space of the form:
\[
Y\cong S^2 \left(0; \left\{\frac{p_1}{q_1}, -\frac{p_1}{q_1}\right\}^{\geq 1}, \dots, \left\{\frac{p_k}{q_k}, -\frac{p_k}{q_k}\right\}^{\geq 1}  \right),
\]
where we assume $1\leq q_i<p_i$ are coprime integers and $\gcd(p_i, p_j)=1$ whenever $i\neq j$. \qed
\end{prop}

\subsection{Constructing embeddings}
In this section we construct embeddings necessary to prove Theorem~\ref{thm:embedding}. Firstly we consider the case where the Seifert fibered space is a $\Z H_*(S^1\times S^2)$. Firstly we note that the existence of the necessary embedding is equivalent to bounding a $\Z H_*(S^1 \times B^3)$.

\begin{lem}\label{lem:doubling_ZHS1xB3}
Let $Y$ be a $\Z H_* (S^1 \times S^2)$. Then there exists a 4-manifold $Z$ such that $Z$ is a $\Z H_*(S^1 \times S^3)$ and $Y$ embeds into $Z$ a $\Z H_*(S^1 \times S^3)$ with the induced map $H_1(Y;\Z) \rightarrow H_1(Z; \Z)$ is surjective if and only if $Y$ bounds a $\Z H_*(S^1 \times B^3)$.
\end{lem}
\begin{proof}
Firstly suppose that $Y$ embeds into $Z$ a $\Z H_*(S^1 \times S^3)$ with the induced map $H_1(Y;\Z) \rightarrow H_1(Z; \Z)$ is surjective. In this case Lemma~\ref{lem:Ui_properties} implies that $Y$ bounds a $\Z H_*(S^1 \times B^3)$. Conversely suppose that $Y$ bounds $X$ a $\Z H_*(S^1 \times S^3)$. The short exact sequence in homology implies that the map $H_1(Y;\Z) \rightarrow H_1(X; \Z)$ is an isomorphism. Now let $Z$ be the manifold obtained by doubling $X$, ie. take $Z = X \cup_Y -X$. By construction, $Z$ contains a copy of $Y$ embedded as $\partial X$. An easy application of the Meyer-Vietoris sequence implies that $Z$ is a $\Z H_*(S^1 \times S^3)$ and the map $H_1(Y;\Z) \rightarrow H_1(Z;\Z)$ is an isomorphism.
\end{proof}

Now we construct $\Z H_*(S^1 \times B^3)$s cobounding Seifert fibered spaces. 
\begin{lem}\label{lem:S1xB3construction}
Let $Y=S^2(0; \frac{p_1}{q_1}, -\frac{p_1}{q_1}, \dots, \frac{p_k}{q_k}, -\frac{p_k}{q_k})$ be Seifert fibered space where $p_i$ and $p_j$ are coprime for all $i \ne j$. Then $Y$ bounds a $\Z H_*(S^1 \times B^3)$. 
\end{lem}
\begin{proof}
By Lemma~\ref{lemma:sfs_homology}, we can compute that $Y$ is an integer homology $S^1 \times S^2$. Now consider the Seifert fibered space $M=S^2(0; \frac{p_1}{q_1}, \dots, \frac{p_k}{q_k})$ and let $M'$ be the Seifert fibered space obtained by deleting an open Seifert fibered neighbourhood of a regular fiber in $M$. One can check by various means that $M'$ is a $\Z H_*(S^1 \times B^2)$. For example we can fill in the toroidal boundary component to obtain an integer homology sphere. Thus let $Z'$ be the product
\[
Z':=M' \times [0,1].
\]
By construction $Z'$ is a $\Z H_*(S^1\times B^3)$. The boundary of $Z'$ is the double of $M'$, which is homeomorphic to $Y$.



\end{proof}
We remind the reader of the definition of expansion for Seifert fibered spaces. Given the Seifert fibered space $Y=S^2(e;\frac{p_1}{q_1}, \dots, \frac{p_k}{q_k})$, we say that a space $Y'$ is obtained from $Y$ by expansion, if it takes the form $Y'=S^2(e;\frac{p_1}{q_1}, \dots, \frac{p_k}{q_k}, -\frac{p_j}{q_j},\frac{p_j}{q_j})$ for some $j$ in the range $1\leq j \leq k$. The following is a refinement of \cite[Lemma~1.6]{Issa2018embedding}. The final step is to observe how embeddings change under expansion.
\begin{lem}\label{lem:expansion}
Let $Y$ and $Y'$ be Seifert fibered spaces such that $Y'$ is obtained from $Y$ by expansion. Then there is a smooth embedding
\[
\iota :Y' \rightarrow Y \times [0,1]
\]
and the induced map
\[
\iota_* : H_1(Y') \rightarrow H_1 (Y \times [0,1]) \cong H_1(Y)
\]
is surjective.
\end{lem}
\begin{proof}
Let $Y=S^2(e;\frac{p_1}{q_1}, \dots, \frac{p_k}{q_k})$ and $Y'=S^2(e;\frac{p_1}{q_1}, \dots, \frac{p_k}{q_k}, -\frac{p_k}{q_k},\frac{p_k}{q_k})$ a space obtained from $Y$ by expansion.
  We will explicitly find a subset of $Y\times[0,1]$ which is homeomorphic to $Y'$. Let $N_1\subset Y$ be a Seifert fibered neighbourhood of the exceptional fiber corresponding to $p_k/q_k$, that is, a set homeomorphic to $S^1 \times B^2$ whose boundary is a union of regular fibers. Consider the set $M=N_1 \times [\frac{1}{4}, \frac{3}{4}]$. The boundary $\partial M$ is homeomorphic to $S^1 \times S^2$ and it naturally inherits a Seifert fibered structure of the form $\partial M =S^2(0; -\frac{p_k}{q_k}, \frac{p_k}{q_k})$. On $N_1 \times \{\frac{1}{4}\}$ and $N_1 \times \{\frac{3}{4}\}$ this structure is a translate of the one on $N_1$, giving the two exceptional fibers, and is the obvious product structure on $\partial N_1 \times [\frac{1}{4}, \frac{3}{4}]$. Now let $N_2 \subseteq N_1$ be a Seifert fibered neighbourhood of a regular fiber. 
  We take $X$ to be following the subset of $Y\times [0,1]$:
  \[
  X=(Y\setminus \inter N_2) \times \{0\} \cup \partial N_2 \times \left[0, \frac{1}{4}\right] \cup \left(\partial M \setminus \inter N_2 \right) \times \left\{\frac{1}{4}\right\}
  \]
  As a manifold, $X$ is obtained by taking $Y$ and $M$, deleting open fibered neighbourhoods of regular fibers in both and gluing the two resulting manifolds along their boundaries so that the boundary fibers match up. From this description $X$ is clearly homeomorphic to $Y'$. Thus by smoothing the corners of $X$ we can obtain a smooth embedding $\iota$ of $Y'$ into $Y\times [0,1]$.
  
  To see that the induced map $\iota_*$ on homology is surjective, observe that any class in $H_1(Y\times \{0\})$ can be represented by a loop $\gamma$. Moreover since $N_2$ is a solid torus, $\gamma$ is homotopic inside $Y\times \{0\}$ to a curve $(Y\setminus \inter N_2) \times \{0\}$. Thus the loop $\gamma$ is contained in $X$ and hence is in the image of $\iota_*$
\end{proof}

 \subsection{Proof of Theorem~\ref{thm:embedding}}

We have now established all the ingredients necessary to prove our main result on embedding Seifert fibered spaces.
\thmSFembedding*
\begin{proof}
Proposition~\ref{prop:obstruction_summary} gives the implication $\eqref{it:embedding} \Rightarrow \eqref{it:explicit_Description}$. The implication $\eqref{it:explicit_Description} \Rightarrow \eqref{it:expansion}$ follows from Lemma~\ref{lem:S1xB3construction} and the definition of expansion. Together Lemma~\ref{lem:doubling_ZHS1xB3} and Lemma~\ref{lem:expansion} finish the proof by establishing the implication $\eqref{it:expansion} \Rightarrow \eqref{it:embedding}$.
\end{proof}

\section{Further questions}
Finally, we discuss a few questions that naturally arise from the results in this paper. 
Firstly, Theorem~\ref{thm:embedding} is the best possible result that can be obtained via our Donaldson's theorem obstruction, since this technique can only obstruct embeddings into $\Z H_*(S^1 \times S^3)$s. This naturally leads one to wonder about the difference between embedding into $\Z H_*(S^1 \times S^3)$s and $S^1\times S^3$. The following question seems like a natural first step in understanding this distinction. 

\begin{question}\label{question:homembedding}
Does there exist a 3-manifold $Y$ that embeds into a $\Z H_*(S^1 \times S^3)$ by a map that induces a surjection on $H_1$ but $Y$ does not admit such an embedding into $S^1\times S^3$?
\end{question}
It seems probable that many of the Seifert fibered spaces from Theorem~\ref{thm:embedding} which are known to embed into $\Z H_*(S^1 \times S^3)$ but do not have any known embedding into $S^1\times S^3$ would be good candidates for an affirmative answer to Question~\ref{question:homembedding}. Relating to this it is natural to wonder about the strong double slicing of the Montesinos links which are unobstructed by Theorem~\ref{thm:montesinos}.

\begin{question}\label{question:remainingcases}
Let $L$ be the Montesinos link
\[
L=\M \left(0; \frac{p_1}{q_1}, \dots,\frac{p_k}{q_k}, -\frac{p_k}{q_k}, \dots, -\frac{p_1}{q_1}  \right),
\]
where at most one of the $p_i$ is even and for all $i$ we have  $|p_i/q_i|\geq 2$ and  for all $i,j$ we have $\gcd(p_i,p_j) = 1$ or  $\frac{p_i}{q_i} = \pm \frac{p_j}{q_j}$. Are $L$ and its mutants strongly doubly slice?
\end{question}
Even if one could fully understand which Seifert fibered spaces embed in $S^1 \times S^3$, we expect that there are many Montesinos links whose double covers embed in $S^1 \times S^3$, but do not seem to be strongly doubly slice. Potentially one could obtain further progress by obstructing embeddings of higher branched covers into connect sums of $S^1\times S^3$. As precedent for this we note that obstructions to slicing of pretzel links have been obtained by studying properties of the triple branched cover \cite{Acetoandco}.

Finally, we note that our techniques are not very sensitive to questions of orientation. Although we were able to exhibit a three component link which is weakly doubly slice with only one quasi-orientation, exhibiting a two component link with this property appears to be considerably harder.
\begin{question}\label{ques:weak-slicing}
Is there a 2-component link that is weakly doubly slice with one quasi-orientation but not the other?
\end{question}
It seems highly likely that Question~\ref{ques:weak-slicing} should have an affirmative answer, and we can even propose some candidate examples. Let $L$ be a 0-framed 2-cable of a knot $K$ which is slice, but not doubly slice. One can construct a weak slicing for $L$ as follows. Take a slice disk $D$ for $K$ and take a tubular neighbourhood $\nu D$. Since this tubular neighbourhood is diffeomorphic to $D \times B^2$, we can find a $D \times I$ subbundle of the tubular neighbourhood. The boundary of this $D \times I$ subbundle is by construction an unknotted $S^2$ and after perturbing slightly it will intersect the equatorial $S^3$ in a copy of $L$. The quasi-orientation induced on $L$ by this slicing will be the one such that the 2-components of $L$ cobound an annulus in $S^3$. However, there does not seem to be any reason to expect that $L$ is weakly doubly slice with the other quasi-orientation.


\bibliography{master}
\bibliographystyle{alpha}
\end{document}